\documentclass[oneside,british]{amsart}
\usepackage[T1]{fontenc}
\usepackage[latin9]{inputenc}
\usepackage{amstext}
\usepackage{amsthm}
\usepackage{amssymb}

\makeatletter
\numberwithin{equation}{section}
\numberwithin{figure}{section}
\theoremstyle{plain}
\newtheorem{thm}{\protect\theoremname}[section]
\theoremstyle{definition}
\newtheorem{defn}[thm]{\protect\definitionname}
\theoremstyle{remark}
\newtheorem{rem}[thm]{\protect\remarkname}
\theoremstyle{plain}
\newtheorem{cor}[thm]{\protect\corollaryname}
\theoremstyle{plain}
\newtheorem{lem}[thm]{\protect\lemmaname}
\theoremstyle{plain}
\newtheorem{prop}[thm]{\protect\propositionname}
\theoremstyle{plain}
\newtheorem*{prop*}{\protect\propositionname}
\theoremstyle{plain}
\newtheorem*{lem*}{\protect\lemmaname}
\theoremstyle{plain}
\newtheorem*{thm*}{\protect\theoremname}

\usepackage[pdftex,pdfpagelabels,bookmarks,hyperindex,hyperfigures]{hyperref}

\usepackage{thmtools}

\makeatother

\usepackage{babel}
\providecommand{\corollaryname}{Corollary}
\providecommand{\definitionname}{Definition}
\providecommand{\lemmaname}{Lemma}
\providecommand{\propositionname}{Proposition}
\providecommand{\remarkname}{Remark}
\providecommand{\theoremname}{Theorem}

\begin{document}

\title{Dimension of diagonal self-affine sets and measures via non-conformal
partitions}

\author{\noindent Ariel Rapaport}



\thanks{This research was supported by the Israel Science Foundation (grant
No. 619/22). The author is a Horev Fellow at the Technion -- Israel
Institute of Technology.}
\begin{abstract}
Let $\Phi:=\left\{ (x_{1},...,x_{d})\rightarrow\left(r_{i,1}x_{1}+a_{i,1},...,r_{i,d}x_{d}+a_{i,d}\right)\right\} _{i\in\Lambda}$
be an affine diagonal IFS on $\mathbb{R}^{d}$. Suppose that for each
$1\le j_{1}<j_{2}\le d$ there exists $i\in\Lambda$ so that $|r_{i,j_{1}}|\ne|r_{i,j_{2}}|$,
and that for each $1\le j\le d$ the IFS $\left\{ t\rightarrow r_{i,j}t+a_{i,j}\right\} _{i\in\Lambda}$
on the real line is exponentially separated. Under these assumptions
we show that the Hausdorff dimension of the attractor of $\Phi$ is
equal to $\min\left\{ \dim_{A}\Phi,d\right\} $, where $\dim_{A}\Phi$
is the affinity dimension. This follows from a result regarding self-affine
measures, which says that, under the additional assumption that the
linear parts of the maps in $\Phi$ are all contained in a $1$-dimensional
subgroup, the dimension of an associated self-affine measure $\mu$
is equal to the minimum of its Lyapunov dimension and $d$.

Most of the proof is dedicated to an entropy increase result for convolutions
of $\mu$ with general measures $\theta$ of non-negligible entropy,
where entropy is measured with respect to non-conformal partitions
corresponding to the Lyapunov exponents of $\mu$. It turns out that
with respect to these partitions, the entropy across scales of repeated
self-convolutions of $\theta$ behaves quite differently compared
to the conformal case. The analysis of this non-conformal multi-scale
entropy is the main ingredient of the proof, and is also the main
novelty of this paper.
\end{abstract}

\maketitle

\section{Introduction}

\subsection{Background and main results}

Fix $d\ge1$ and let $\Phi=\{\varphi_{i}(x)=A_{i}x+a_{i}\}_{i\in\Lambda}$
be a finite collection of invertible affine contractions of $\mathbb{R}^{d}$.
Such a collection is called an affine iterated function system (IFS).
It is well known (see \cite{Hut}) that there exists a unique nonempty
compact $K_{\Phi}\subset\mathbb{R}^{d}$ with $K_{\Phi}=\cup_{i\in\Lambda}\varphi_{i}(K_{\Phi})$.
It is called the attractor or self-affine set corresponding to $\Phi$.

Many mathematical problems surround self-affine sets, but perhaps
the most natural one is to determine their dimension. It has been
studied by many authors, and its computation is one of the major open
problems in fractal geometry. In what follows, we denote the Hausdorff
dimension of $E\subset\mathbb{R}^{d}$ by $\dim_{H}E$.

In the 1980's, Falconer \cite{falconer1988hausdorff} introduced a
natural upper bound for the dimension of $K_{\Phi}$, which is called
the affinity dimension. It is denoted by $\dim_{A}\Phi$ and depends
only on the linear parts $\{A_{i}\}_{i\in\Lambda}$. Falconer has
shown that when $\Vert A_{i}\Vert_{op}<1/2$ for $i\in\Lambda$, and
under a natural randomisation of the translations $\{a_{i}\}_{i\in\Lambda}$,
the equality
\begin{equation}
\dim_{H}K_{\Phi}=\min\{d,\dim_{A}\Phi\}\label{eq:dim K =00003D dim_A}
\end{equation}
holds almost surely. (In fact, Falconer proved this with $1/3$ as
the upper bound on the norms; it was subsequently shown by Solomyak
\cite{So} that $1/2$ suffices.) The definition of $\dim_{A}\Phi$
in the setting studied in this paper is provided in Section \ref{sec:Proof-of-main thm for sets}.

The last result shows that (\ref{eq:dim K =00003D dim_A}) holds typically,
but does not provide any explicit examples. It is of course desirable
to find explicit and verifiable conditions under which equality holds.
In recent years, and while assuming the collection $\{A_{i}\}_{i\in\Lambda}$
is strongly irreducible, such conditions have been obtained when $d=2$
(see \cite{BHR,HR,MoSh}) and when $d=3$ (see \cite{MoSe,Rap_SA_Rd}).

An important subclass of self-affine systems, which is in a sense
opposite to the strongly irreducible case, is the one in which the
linear parts $\{A_{i}\}_{i\in\Lambda}$ are all diagonal. The goal
of this paper is to establish (\ref{eq:dim K =00003D dim_A}) under
mild assumptions in the diagonal case.

For each $i\in\Lambda$ and $1\le j\le d$ let $0\ne r_{i,j}\in(-1,1)$
and $a_{i,j}\in\mathbb{R}$. We assume that $A_{i}=\mathrm{diag}(r_{i,1},...,r_{i,d})$
and $a_{i}=(a_{i,1},...,a_{i,d})$ for $i\in\Lambda$. That is,
\begin{equation}
\varphi_{i}(x)=\left(r_{i,1}x_{1}+a_{i,1},...,r_{i,d}x_{d}+a_{i,d}\right)\text{ for }(x_{1},...,x_{d})=x\in\mathbb{R}^{d}.\label{eq:form of maps}
\end{equation}
Note that $\Phi_{j}:=\left\{ t\rightarrow r_{i,j}t+a_{i,j}\right\} _{i\in\Lambda}$
is an affine IFS on $\mathbb{R}$ for each $1\le j\le d$. In order
to state our results we shall need the following definition.
\begin{defn}
\label{def:exp sep}Given affine maps $\psi_{1},\psi_{2}:\mathbb{R}\rightarrow\mathbb{R}$
with $\psi_{i}(t)=s_{i}t+b_{i}$ for $i=1,2$, we write
\[
\rho(\psi_{1},\psi_{2}):=\begin{cases}
\infty & \text{ if }s_{1}\ne s_{2}\\
|b_{1}-b_{2}| & \text{ else}
\end{cases}.
\]
An affine IFS $\Psi:=\{\psi_{i}\}_{i\in\Lambda}$ on $\mathbb{R}$
is said to be exponentially separated if there exist $c>0$ and an
infinite $\mathrm{Q}\subset\mathbb{Z}_{>0}$ so that $\rho(\psi_{u_{1}},\psi_{u_{2}})\ge c^{n}$
for all $n\in\mathrm{Q}$ and distinct $u_{1},u_{2}\in\Lambda^{n}$,
where $\psi_{u}:=\psi_{i_{1}}\circ...\circ\psi_{i_{n}}$ for $i_{1}...i_{n}=u\in\Lambda^{n}$.
It is said that $\Psi$ has no exact overlaps if $\psi_{u_{1}}\ne\psi_{u_{2}}$
for all distinct $u_{1},u_{2}\in\Lambda^{*}$, where $\Lambda^{*}$
is the set of finite words over $\Lambda$.
\end{defn}

\begin{rem}
\label{rem:exp set vs exa overl}It is easy to see that $\Psi$ has
no exact overlaps whenever it is exponentially separated. On the other
hand, $\Psi$ is exponentially separated whenever it is defined by
algebraic parameters and has no exact overlaps (see \cite{Ho1}).
\end{rem}

We can now state our main result.
\begin{thm}
\label{thm:main result for sets}Suppose that,
\begin{enumerate}
\item \label{enu:cond reg mod of r_i,j}for each $1\le j_{1}<j_{2}\le d$
there exists $i\in\Lambda$ so that $|r_{i,j_{1}}|\ne|r_{i,j_{2}}|$;
\item \label{enu:exp sep of 1-dim sys}$\Phi_{j}$ is exponentially separated
for each $1\le j\le d$.
\end{enumerate}
Then $\dim_{H}K_{\Phi}=\min\left\{ \dim_{A}\Phi,d\right\} $.
\end{thm}

\begin{rem}
\label{rem:past results for sets}In the case $d=1$ the theorem was
obtained by Hochman \cite{Ho1}. When $d=2$ and the similarity dimension
of either $\Phi_{1}$ or $\Phi_{2}$ is at most $1$, Bárány, Rams
and Simon \cite{BaRaSi} have applied Hochman's work in order to deduce
the theorem.
\end{rem}

\begin{rem}
It is not difficult to modify \cite[Example 1.2]{Ho} in order to
construct a homothetic IFS $\Psi:=\left\{ \psi_{i}(x)=rx+(b_{i,1},b_{i,2})\right\} _{i\in I}$
on $\mathbb{R}^{2}$ with attractor $K_{\Psi}$, such that $\left\{ t\rightarrow rt+b_{i,j}\right\} _{i\in I}$
is exponentially separated for $j=1,2$, $K_{\Psi}$ is not contained
in an affine line, and $\dim_{H}K_{\Psi}<\min\left\{ \dim_{A}\Psi,2\right\} $.
This shows that Condition (\ref{enu:cond reg mod of r_i,j}) cannot
be dropped from the statement of the theorem.
\end{rem}

\begin{rem}
As demonstrated by various carpet-like examples (see e.g. \cite{MR2298824,Bed,MR2947936,MR1183358,Mcm}),
it is necessary to assume that the systems $\Phi_{j}$ have no exact
overlaps. Ideally, one would like to prove Theorem \ref{thm:main result for sets}
under this assumption instead of Condition (\ref{enu:exp sep of 1-dim sys}).
Unfortunately, at this point, this is well beyond our reach. Indeed,
this has not been achieved even when $d=1$, in which case the validity
of such a statement is considered one of the major open problems in
fractal geometry (see \cite{MR3966837,Var_ICM}). Note however that
by Remark \ref{rem:exp set vs exa overl}, Condition (\ref{enu:exp sep of 1-dim sys})
always holds whenever $\Phi$ is defined by algebraic parameters and
$\Phi_{j}$ has no exact overlaps for $1\le j\le d$.
\end{rem}

Theorem \ref{thm:main result for sets} will follow from a result
regarding self-affine measures. Let $p=(p_{i})_{i\in\Lambda}$ be
a probability vector. It is well known (see \cite{Hut}) that there
exists a unique Borel probability measure $\mu$ on $\mathbb{R}^{d}$
which satisfies the relation $\mu=\sum_{i\in\Lambda}p_{i}\cdot\varphi_{i}\mu$,
where $\varphi_{i}\mu$ is the push-forward of $\mu$ via $\varphi_{i}$.
It is supported on $K_{\Phi}$, and is called the self-affine measure
corresponding to $\Phi$ and $p$.

In recent years, it has been proven that self-affine measures are
always exact dimensional (see \cite{FH-dimension} for the diagonal
case and \cite{BK,MR4557759} for the general case). That is, there
exists a number $\dim\mu$ so that
\[
\underset{\delta\downarrow0}{\lim}\:\frac{\log\mu(B(x,\delta))}{\log\delta}=\dim\mu\text{ for }\mu\text{-a.e. }x,
\]
where $B(x,\delta)$ is the closed ball with centre $x$ and radius
$\delta$.

Recall that we assume that the maps $\varphi_{i}$ are of the form
(\ref{eq:form of maps}). For $1\le j\le d$ set $\chi_{j}:=-\sum_{i\in\Lambda}p_{i}\log|r_{i,j}|$.
It is easy to verify that $\chi_{1},...,\chi_{d}$ are equal to the
Lyapunov exponents corresponding to $\{A_{i}\}_{i\in\Lambda}$ and
$p$ (as defined e.g. in \cite[Section III.5]{BL}).

As in the case of self-affine sets, there exists a natural upper bound
for the dimension of $\mu$. It is denoted by $\dim_{L}(\Phi,p)$,
is called the Lyapunov dimension corresponding to $\Phi$ and $p$,
and depends only on $\chi_{1},...,\chi_{d}$ and the entropy of $p$.
It has been shown by Jordan, Pollicott and Simon \cite{JPS} that
when $\Vert A_{i}\Vert_{op}<1/2$ for $i\in\Lambda$, and under a
natural randomisation of the translations $\{a_{i}\}_{i\in\Lambda}$,
the equality $\dim\mu=\min\{d,\dim_{L}(\Phi,p)\}$ holds almost surely.
The definition of $\dim_{L}(\Phi,p)$ is provided in Section \ref{subsec:Proof-of-main Theorem for measures}
below.

We shall say that the linear parts of $\Phi$ are contained in a $1$-dimensional
subgroup if there exist $c_{1},...,c_{d}\in\mathbb{R}_{>0}$ so that,
\[
\left(|r_{i,1}|,...,|r_{i,d}|\right)\in\left\{ \left(c_{1}^{t},...,c_{d}^{t}\right)\::\:t\in\mathbb{R}\right\} \text{ for each }i\in\Lambda.
\]
In particular this is the case when $\Phi$ is homogeneous, i.e. when
$A_{i_{1}}=A_{i_{2}}$ for all $i_{1},i_{2}\in\Lambda$. We can now
state our main result regarding self-affine measures.
\begin{thm}
\label{thm:main result for measures}Suppose that,
\begin{enumerate}
\item \label{enu:cond dist LY-exp}$\chi_{j_{1}}\ne\chi_{j_{2}}$ for all
$1\le j_{1}<j_{2}\le d$;
\item $\Phi_{j}$ is exponentially separated for each $1\le j\le d$;
\item \label{enu:cond reg 1-dim sub}the linear parts of $\Phi$ are contained
in a $1$-dimensional subgroup.
\end{enumerate}
Then $\dim\mu=\min\left\{ \dim_{L}(\Phi,p),d\right\} $.
\end{thm}

\begin{rem}
The author expects Theorem \ref{thm:main result for measures} to
remain true even without assuming Condition (\ref{enu:cond reg 1-dim sub}).
Unfortunately, our argument crucially depends on this assumption,
and so it cannot be omitted at this point.
\end{rem}

\subsection{A concrete example}

Write,
\[
\Omega_{d}:=\left\{ (r_{1},...,r_{d})\in(0,1)^{d}\::\:r_{j}>r_{j+1}\text{ for }1\le j<d\right\} .
\]
For $(r_{1},...,r_{d})=\bar{r}\in\Omega_{d}$ set
\[
\Phi_{\bar{r}}:=\left\{ (x_{1},...,x_{d})\rightarrow(r_{1}x_{1},...,r_{d}x_{d})\pm(1,..,1)\right\} ,
\]
so that $\Phi_{\bar{r}}$ is an affine IFS on $\mathbb{R}^{d}$ consisting
of two maps. Write $K_{\bar{r}}$ for the attractor of $\Phi_{\bar{r}}$,
and $\mu_{\bar{r}}$ for the self-affine measure corresponding to
$\Phi_{\bar{r}}$ and the probability vector $(1/2,1/2)$. When $d=2$,
these fractal objects have been studied by many authors throughout
the years (see e.g. \cite{MR3590500,MR1301464,MR1002918,MR2269414}
and the references therein). The sets $K_{\bar{r}}$ are perhaps the
simplest nontrivial self-affine sets on $\mathbb{R}^{d}$ which are
not self-similar. The measures $\mu_{\bar{r}}$ can be viewed as a
higher dimensional analog of the family of Bernoulli convolutions.

For $(r_{1},...,r_{d})=\bar{r}\in\Omega_{d}$ set,
\[
m_{\bar{r}}:=\max\left\{ 0\le k\le d\::\:\Pi_{j=1}^{k}r_{j}\ge1/2\right\} .
\]
It follows easily from the definitions of the affinity and Lyapunov
dimensions that with $m=m_{\bar{r}}$,
\[
\dim_{A}\Phi_{\bar{r}}=\dim_{L}\left(\Phi_{\bar{r}},(1/2,1/2)\right)=\begin{cases}
m+\frac{\log2+\sum_{j=1}^{m}\log r_{j}}{-\log r_{m+1}} & ,\text{ if }m<d\\
d\frac{\log2}{-\sum_{j=1}^{d}\log r_{j}} & ,\text{ if }m=d
\end{cases}.
\]

It is desirable to find conditions under which,
\begin{equation}
\dim_{H}K_{\bar{r}}=\dim\mu_{\bar{r}}=\min\left\{ \dim_{A}\Phi_{\bar{r}},d\right\} .\label{eq:exp eq in examp}
\end{equation}
It is easy to verify that (\ref{eq:exp eq in examp}) holds for all
$\bar{r}\in\Omega_{d}$ with $r_{1}\le1/2$. In the planar case $d=2$,
Shmerkin \cite{MR2269414} has proven that $\dim_{H}K_{\bar{r}}=\dim_{A}\Phi_{\bar{r}}$
for $\mathcal{L}_{2}$-a.e. $(r_{1},r_{2})=\bar{r}\in\Omega_{2}$
with $r_{1}r_{2}<1/2<r_{1}$, where $\mathcal{L}_{2}$ is the Lebesgue
measure. By applying the results of the previous section, it is possible
to improve this statement. In what follows, we denote the packing
dimension of $E\subset\mathbb{R}^{d}$ by $\dim_{p}E$. Recall that
the inequality $\dim_{H}E\le\dim_{p}E$ is always satisfied.
\begin{cor}
There exists $F\subset(0,1)^{d}$ with $\dim_{p}F=d-1$ so that (\ref{eq:exp eq in examp})
holds for all $\bar{r}\in\Omega_{d}\setminus F$.
\end{cor}

\begin{proof}
Let $E$ be the set of all $0<r<1$ for which the IFS $\left\{ t\rightarrow rt\pm1\right\} $
on $\mathbb{R}$ is not exponentially separated. As explained in the
proof of \cite[Theorem 1.9]{Ho1}, it holds that $\dim_{p}E=0$. Let
$F$ be the set of all $(r_{1},...,r_{d})\in(0,1)^{d}$ for which
$r_{j}\in E$ for some $1\le j\le d$. From $\dim_{p}E=0$ it follows
easily that $\dim_{p}F=d-1$. The corollary now follows directly from
Theorems \ref{thm:main result for sets} and \ref{thm:main result for measures}.
\end{proof}
\begin{rem}
Shmerkin \cite{MR2269414} has also shown that there exists a nonempty
open set $U\subset\left\{ (r_{1},r_{2})\in\Omega_{2}\::\:1/2<r_{1}r_{2}\right\} $
so that $\mu_{\bar{r}}\ll\mathcal{L}_{2}$ for $\mathcal{L}_{2}$-a.e.
$\bar{r}\in U$. By combining Theorem \ref{thm:main result for measures}
together with a result of Solomyak, it should be possible to improve
this result as well. Indeed, by \cite[Theorem 1.2]{MR4461215} it
follows that $\mu_{\bar{r}}$ has power Fourier decay for $\mathcal{L}_{d}$-a.e.
$\bar{r}\in\Omega_{d}$. By applying this together with Theorem \ref{thm:main result for measures},
it seems possible to use the argument appearing in the proof of \cite[Theorem 1.2]{MR3213835}
in order to show that $\mu_{\bar{r}}\ll\mathcal{L}_{d}$ for $\mathcal{L}_{d}$-a.e.
$(r_{1},...,r_{d})\in\Omega_{d}$ with $\Pi_{j=1}^{d}r_{j}>1/2$.
We do not pursue this however.
\end{rem}

We finish this subsection with a statement valid for concrete parameters.
Write $\mathcal{P}_{\{\pm1,0\}}$ for the set of polynomials in one
variable with coefficients $\pm1$ and $0$. Given $0<r<1$, it is
easy to see that the IFS $\left\{ t\rightarrow rt\pm1\right\} $ has
exact overlaps if and only if $r$ is a root of some nonzero $P\in\mathcal{P}_{\{\pm1,0\}}$.
The following corollary follows directly from this fact, from Remark
\ref{rem:exp set vs exa overl}, and from Theorems \ref{thm:main result for sets}
and \ref{thm:main result for measures}.
\begin{cor}
The equalities in (\ref{eq:exp eq in examp}) hold for every $(r_{1},...,r_{d})=\bar{r}\in\Omega_{d}$
so that $r_{1},...,r_{d}$ are algebraic numbers and $P(r_{j})\ne0$
for all nonzero $P\in\mathcal{P}_{\{\pm1,0\}}$ and $1\le j\le d$.
\end{cor}

\subsection{\label{subsec:About-the-proof}About the proof}

Theorem \ref{thm:main result for sets} will be obtained as a corollary
of Theorem \ref{thm:main result for measures}. In this subsection
we discuss the proof of Theorem \ref{thm:main result for measures}.
The proof is carried out by induction on the dimension $d$ of the
ambient space. Thus, suppose that the theorem in known to hold for
all $d'\ge1$ with $d'<d$.

Write $[d]:=\{1,...,d\}$ and for $J\subset[d]$ denote by $\pi_{J}$
the orthogonal projection onto the linear span of $\{e_{j}\}_{j\in J}$,
where $\{e_{1},...,e_{d}\}$ is the standard basis of $\mathbb{R}^{d}$.
Note that $\pi_{J}\mu$ is (an embedded copy of) a self-affine measure
on $\mathbb{R}^{|J|}$ for which the conditions of Theorem \ref{thm:main result for measures}
are also satisfied. Based on this observation, the induction hypothesis,
the Ledrappier-Young formula for self-affine measures (see Section
\ref{subsec:Ledrappier-Young-formula}), and by an argument inspired
by ideas from \cite{BaRaSi} (see Remark \ref{rem:past results for sets}),
it is not difficult to prove that $\dim\mu=\dim_{L}(\Phi,p)$ whenever
there exists a proper subset $J$ of $[d]$ so that $\dim\pi_{J}\mu<|J|$.
Thus, suppose that
\begin{equation}
\dim\pi_{J}\mu=|J|\text{ for every proper subset \ensuremath{J} of \ensuremath{[d]}.}\label{eq:assump reg proj in sketch}
\end{equation}

By assumption the numbers $\chi_{1},...,\chi_{d}$ are distinct. Without
loss of generality, assume that $\chi_{j}<\chi_{j+1}$ for $1\le j<d$.
For each $n\ge0$ define the following non-conformal partition of
$\mathbb{R}^{d}$,
\[
\mathcal{E}_{n}:=\left\{ D_{1}\times...\times D_{d}\::\:D_{j}\in\mathcal{D}_{\chi_{j}n}^{1}\text{ for }1\le j\le d\right\} .
\]
Here, $\mathcal{D}_{t}^{1}$ denotes the level-$\left\lfloor t\right\rfloor $
dyadic partition of $\mathbb{R}$ for $t\in\mathbb{R}_{\ge0}$. Write,
\[
\kappa:=\chi_{d}\dim\mu-\sum_{j=1}^{d-1}(\chi_{d}-\chi_{j}).
\]
As we will see in Section \ref{sec:Asymptotic-entropies-of mu}, from
$\dim\pi_{[d-1]}\mu=d-1$ it follows that
\begin{equation}
\underset{n}{\lim}\:\frac{1}{n}H\left(\mu,\mathcal{E}_{n}\right)=\kappa,\label{eq:=00003Dkappa}
\end{equation}
where $H\left(\mu,\mathcal{E}_{n}\right)$ is the entropy of $\mu$
with respect to $\mathcal{E}_{n}$.

By employing ideas of Hochman, which are used in \cite{Ho1} for the
case $d=1$, we will see that, in order to complete the induction
(and hence the proof), it suffices to prove the following entropy
increase result. In what follows we write $\mathcal{M}_{\mathrm{c}}(\mathbb{R}^{d})$
for the collection of compactly supported Borel probability measures
on $\mathbb{R}^{d}$.
\begin{thm}
\label{thm:ent inc result}Suppose that $\dim\mu<d$ and that $\dim\pi_{J}\mu=|J|$
for each proper subset $J$ of $[d]$. Then for every $0<\epsilon<1$
there exists $\delta=\delta(\epsilon)>0$ so that the following holds.
Let $n\ge N(\epsilon)\ge1$ and $\theta\in\mathcal{M}_{\mathrm{c}}(\mathbb{R}^{d})$
be with $\mathrm{diam}(\mathrm{supp}(\theta))\le\epsilon^{-1}$ and
$\frac{1}{n}H(\theta,\mathcal{E}_{n})>\epsilon$. Then $\frac{1}{n}H(\theta*\mu,\mathcal{E}_{n})>\kappa+\delta$.
\end{thm}

For the rest of this subsection we discuss the proof of Theorem \ref{thm:ent inc result}.
The proof requires the following two lemmas, which concern non-conformal
conditional entropies of $\mu$. Recall that the linear parts of $\Phi$
are assumed to be contained in a $1$-dimensional subgroup. The derivation
of (\ref{eq:=00003Dkappa}) and the proofs of Lemmas \ref{lem:lb on ent of comp}
and \ref{lem:lb on proj of comp} are the only places in the argument
where this assumption is used.
\begin{lem}
\label{lem:lb on ent of comp}Suppose that $\dim\pi_{[d-1]}\mu=d-1$.
Then for all $\epsilon>0$, $m\ge M(\epsilon)\ge1$ and $n\ge1$,
\[
\frac{1}{m}H\left(\mu,\mathcal{E}_{n+m}\mid\mathcal{E}_{n}\right)\ge\kappa-\epsilon.
\]
\end{lem}

\begin{lem}
\label{lem:lb on proj of comp}Let $J\subset[d]$ be with $\dim\pi_{J}\mu=|J|$.
Then for all $\epsilon>0$, $m\ge M(\epsilon)\ge1$ and $n\ge1$,
\[
\frac{1}{m}H\left(\mu,\pi_{J}^{-1}\mathcal{E}_{n+m}\mid\mathcal{E}_{n}\right)\ge\sum_{j\in J}\chi_{j}-\epsilon.
\]
\end{lem}

In the last lemma, the notation $\pi_{J}^{-1}\mathcal{E}_{n+m}$ means
$\left\{ \pi_{J}^{-1}(E)\::\:E\in\mathcal{E}_{n+m}\right\} $. The
validity of this lemma is one of main reasons why it is necessary
to work with the non-conformal partitions instead of standard dyadic
partitions. Indeed, it is not clear how to prove this statement with
$\mathcal{D}_{n+m}^{d}$ and $\mathcal{D}_{n}^{d}$ in place of $\mathcal{E}_{n+m}$
and $\mathcal{E}_{n}$, or even if the statement remains true after
this modification.

The main ingredient in the proof of Theorem \ref{thm:ent inc result}
is the following proposition, which is probably the main novelty of
this paper. For $\theta\in\mathcal{M}_{\mathrm{c}}(\mathbb{R}^{d})$
and $k\ge1$ we write $\theta^{*k}$ for the $k$-fold convolution
of $\theta$ with itself. Given $R_{1},R_{2}\ge1$ we write $R_{1}\ll R_{2}$
in order to indicate that $R_{2}$ is large with respect to $R_{1}$.
We further elaborate on this notation in Section \ref{subsec:Basic-notations}.
\begin{prop}
\label{prop:full ent of slices}For every $0<\epsilon<1$ there exists
$\delta>0$ so that the following holds. Let $m_{1},...,m_{d},k_{1},...,k_{d},n\in\mathbb{Z}_{>0}$
be such that,
\begin{itemize}
\item $\epsilon^{-1}\ll m_{d}$ and $k_{1}\ll n$;
\item $m_{j}\ll k_{j}$ for $1\le j\le d$;
\item $k_{j+1}\ll m_{j}$ for $1\le j<d$.
\end{itemize}
Let $\theta\in\mathcal{M}_{\mathrm{c}}(\mathbb{R}^{d})$ be with $\mathrm{diam}(\mathrm{supp}(\theta))\le\epsilon^{-1}$
and $\frac{1}{n}H(\theta,\mathcal{E}_{n})>\epsilon$. For $1\le j\le d$
write $\mathrm{Q}_{j}$ for the set of all integers $1\le q\le n$
so that,
\[
\frac{1}{m_{j}}H\left(\theta^{*k_{j}},\mathcal{E}_{q+m_{j}}\mid\mathcal{E}_{q}\vee\pi_{[d]\setminus\{j\}}^{-1}\mathcal{E}_{q+m_{j}}\right)>\chi_{j}-\epsilon.
\]
Then there exists $1\le j\le d$ such that $\frac{1}{n}\left|\mathrm{Q}_{j}\right|>\delta$.
\end{prop}

Given $\sigma\in\mathcal{M}_{\mathrm{c}}(\mathbb{R}^{d})$, $q\in\mathbb{Z}_{>0}$
and $E\in\mathcal{E}_{q}$ with $\sigma(E)>0$, we refer to $\sigma_{E}:=\frac{1}{\sigma(E)}\sigma|_{E}$
as a (non-conformal) $q$-component of $\sigma$. Roughly speaking,
the conclusion of the proposition is that there exists $1\le j\le d$
so that, for a nonnegligible proportion of scales $1\le q\le n$,
the entropy of most $q$-components of $\theta^{*k_{j}}$ along narrow
rectangular tubes in direction $e_{j}\mathbb{R}$ is almost full.
It is easy to see that the proposition becomes false if the partitions
$\mathcal{E}_{n}$ are replaced with the conformal partitions $\mathcal{D}_{n}^{d}$.
The proof of the proposition relies on the Berry-Esseen theorem, and
is inspired by the proof of Hochman's inverse theorem for the entropy
of convolutions on $\mathbb{R}$ \cite[Theorem 2.7]{Ho1}.

Given the last three statements, it is not difficult to complete the
proof of Theorem \ref{thm:ent inc result}. Let $\epsilon$, $n$
and $\theta$ be as in the statement of the theorem. By using Proposition
\ref{prop:full ent of slices}, Lemma \ref{lem:lb on proj of comp},
and (\ref{eq:assump reg proj in sketch}), it is possible to show
that there exist $m,k\in\mathbb{Z}_{>0}$ with $\epsilon^{-1}\ll m\ll k\ll n$
so that $h_{i}:=\frac{1}{m}H\left(\theta^{*k}*\mu,\mathcal{E}_{i+m}\mid\mathcal{E}_{i}\right)$
is close to its maximal possible value $\sum_{j=1}^{d}\chi_{j}$ for
a nonnegligible proportion of scales $1\le i\le n$. Note that from
$\dim\mu<d$ it follows that $\kappa$ is strictly smaller than $\sum_{j=1}^{d}\chi_{j}$.
At the same time, Lemma \ref{lem:lb on ent of comp} implies that
$h_{i}$ is (almost) bounded from below by $\kappa$ for all scales
$1\le i\le n$. Together, these facts suggest that $\frac{1}{n}H(\theta^{*k}*\mu,\mathcal{E}_{n})>\kappa+\delta_{0}$
for some $0<\delta_{0}<1$ with $\epsilon^{-1}\ll\delta_{0}^{-1}\ll m$.
By a lemma of Kaimanovich and Vershik \cite{kaimanovich_and_vershik},
it now follows that $\frac{1}{n}H(\theta*\mu,\mathcal{E}_{n})>\kappa+\delta_{0}/(2k)$,
which completes the proof of the theorem with $\delta:=\delta_{0}/(2k)$.

\subsubsection*{\textbf{\emph{Structure of the paper}}}

In Section \ref{sec:Preliminaries}, we develop necessary notations
and background. Section \ref{sec:Entropy-of-repeated-conv} is devoted
to the proof of Proposition \ref{prop:full ent of slices}. In Section
\ref{sec:Asymptotic-entropies-of mu}, we prove Lemmata \ref{lem:lb on ent of comp}
and \ref{lem:lb on proj of comp}. In Section \ref{sec:Proof-of-ent-unc-res},
we establish our entropy increase result. In Section \ref{sec:Proof-of main thm for measures},
we prove Theorem \ref{thm:main result for measures}. Finally, in
Section \ref{sec:Proof-of-main thm for sets}, we deduce Theorem \ref{thm:main result for sets}.

\subsubsection*{\textbf{\emph{Acknowledgment}}}

I would like to thank Thomas Jordan for helpful discussions.

\section{\label{sec:Preliminaries}Preliminaries}

\subsection{\label{subsec:Basic-notations}Basic notations}

In what follows, the base of the logarithm is always $2$. For an
integer $k\ge0$ we write $[k]$ in place of $\{1,...,k\}$, so that
$[0]=\emptyset$. When $k\ge1$, we denote by $\lambda_{k}$ the normalized
counting measure on $[k]$. That is, $\lambda_{k}\{i\}=k^{-1}$ for
each $1\le i\le k$.

Given a metric space $X$, denote by $\mathcal{M}(X)$ the collection
of all Borel probability measures on $X$. We write $\mathcal{M}_{\mathrm{c}}(X)$
for the set of compactly supported members of $\mathcal{M}(X)$.

Given $R_{1},R_{2}\in\mathbb{R}$ with $R_{1},R_{2}\ge1$, we write
$R_{1}\ll R_{2}$ in order to indicate that $R_{2}$ is large with
respect to $R_{1}$. Formally, this means that $R_{2}\ge f(R_{1})$,
where $f$ is an unspecified function from $[1,\infty)$ into itself.
The values attained by $f$ are assumed to be sufficiently large in
a manner depending on the specific context.

Similarly, given $0<\epsilon_{1},\epsilon_{2}<1$ we write $R_{1}\ll\epsilon_{1}^{-1}$,
$\epsilon_{2}^{-1}\ll R_{2}$ and $\epsilon_{1}^{-1}\ll\epsilon_{2}^{-1}$
in order to respectively indicate that $\epsilon_{1}$ is small with
respect to $R_{1}$, $R_{2}$ is large with respect to $\epsilon_{2}$,
and $\epsilon_{2}$ is small with respect to $\epsilon_{1}$.

The relation $\ll$ is clearly transitive. That is, if $R_{1}\ll R_{2}$
and for $R_{3}\ge1$ we have $R_{2}\ll R_{3}$, then also $R_{1}\ll R_{3}$.
For instance, the sentence `Let $m\ge1$, $k\ge K(m)\ge1$ and $n\ge N(m,k)\ge1$
be given' is equivalent to `Let $m,k,n\ge1$ be with $m\ll k\ll n$'.

\subsection{\label{subsec:The-setup}The setup}

Fix a finite nonempty index set $\Lambda$, and let $d\ge1$ be an
integer. For each $i\in\Lambda$ and $1\le j\le d$ let $0\ne r_{i,j}\in(-1,1)$
and $a_{i,j}\in\mathbb{R}$, and let $\varphi_{i,j}:\mathbb{R}\rightarrow\mathbb{R}$
be with $\varphi_{i,j}(t)=r_{i,j}t+a_{i,j}$ for $t\in\mathbb{R}$.

Given $\emptyset\ne J\subset[d]$ we denote by $\Phi_{J}$ the IFS
$\{\varphi_{i,J}\}_{i\in\Lambda}$ on $\mathbb{R}^{J}$, where
\[
\varphi_{i,J}(x):=(\varphi_{i,j}(x_{j}))_{j\in J}\text{ for }i\in\Lambda\text{ and }(x_{j})_{j\in J}=x\in\mathbb{R}^{J}.
\]
For $1\le j\le d$ we write $\Phi_{j}$ in place of $\Phi_{\{j\}}$.
We also write $\Phi$ in place of $\Phi_{[d]}$, and $\varphi_{i}$
in place of $\varphi_{i,[d]}$ for $i\in\Lambda$.

Denote by $K_{\Phi}$ the self-affine set corresponding to $\Phi$,
and by $\Pi:\Lambda^{\mathbb{N}}\rightarrow K_{\Phi}$ the coding
map corresponding to $\Phi$. That is,
\[
\Pi(\omega):=\underset{n\rightarrow\infty}{\lim}\:\varphi_{\omega_{0}}\circ...\circ\varphi_{\omega_{n}}(0)\text{ for }(\omega_{k})_{k\ge0}=\omega\in\Lambda^{\mathbb{N}}.
\]

From now on, until the end of Section \ref{sec:Proof-of main thm for measures},
fix a probability vector $p=(p_{i})_{i\in\Lambda}$. Let $\mu\in\mathcal{M}_{\mathrm{c}}(\mathbb{R}^{d})$
be the self-affine measure corresponding to $\Phi$ and $p$. For
$1\le j\le d$ set $\chi_{j}:=-\sum_{i\in\Lambda}p_{i}\log|r_{i,j}|$,
and write $\chi:=(\chi_{1},...,\chi_{d})$.

Until the end of Section \ref{sec:Proof-of main thm for measures},
we shall always assume that Conditions (\ref{enu:cond dist LY-exp})
and (\ref{enu:cond reg 1-dim sub}) of Theorem \ref{thm:main result for measures}
are satisfied. By Condition (\ref{enu:cond dist LY-exp}), without
loss of generality, we may assume that $\chi_{j}<\chi_{j+1}$ for
$1\le j<d$. Moreover, by Condition (\ref{enu:cond reg 1-dim sub}),
\begin{equation}
\left(-\log|r_{i,1}|,...,-\log|r_{i,d}|\right)\in\left\{ t\chi\::\:t\in\mathbb{R}_{>0}\right\} \text{ for }i\in\Lambda.\label{eq:diag mat in 1-dim subgroup}
\end{equation}

\subsection{Algebraic notations}

Let $\Lambda^{*}$ be the set of finite words over $\Lambda$. Given
a group $G$, indexed elements $\{g_{i}\}_{i\in\Lambda}\subset G$,
and a word $i_{1}...i_{n}=u\in\Lambda^{*}$, we often write $g_{u}$
in place of $g_{i_{1}}\cdot...\cdot g_{i_{n}}$. For the empty word
$\emptyset$ we write $g_{\emptyset}$ in place of $1_{G}$, where
$1_{G}$ is the identity of $G$.

Let $\{e_{1},...,e_{d}\}$ be the standard basis of $\mathbb{R}^{d}$.
Given $J\subset[d]$ denote by $\pi_{J}$ the orthogonal projection
onto $\mathrm{span}\{e_{j}\::\:j\in J\}$. Thus,
\[
\pi_{J}(x)=\sum_{j\in J}\left\langle e_{j},x\right\rangle e_{j}\text{ for }x\in\mathbb{R}^{d},
\]
and $\pi_{[0]}$ is identically $0$. We write $\pi_{j}$ in place
of $\pi_{\{j\}}$ for $1\le j\le d$. Given $\emptyset\ne J\subset[d]$,
note that $\pi_{J}\mu$ is (an embedded copy of) the self-affine measure
corresponding to $\Phi_{J}$ and $p$, where $\pi_{J}\mu:=\mu\circ\pi_{J}^{-1}$
is the push-forward of $\mu$ via $\pi_{J}$. In particular, $\pi_{J}\mu$
is exact dimensional.

For $\theta\in\mathcal{M}(\mathbb{R}^{d})$ and $k\ge1$ we write
$\theta^{*k}$ for the $k$-fold convolution of $\theta$ with itself,
and $\theta^{\times k}$ for the $k$-fold product of $\theta$ with
itself. Thus, for every bounded measurable $f:\mathbb{R}^{d}\rightarrow\mathbb{R}$,
\[
\int f\:d\theta^{*k}=\int f(x_{1}+..+x_{k})\:d\theta^{\times k}(x_{1},...,x_{k}).
\]

Given $\alpha_{1},...,\alpha_{d}\in\mathbb{R}$ we denote by $\mathrm{diag}(\alpha_{1},...,\alpha_{d})$
the linear map $D:\mathbb{R}^{d}\rightarrow\mathbb{R}^{d}$ with $De_{j}=\alpha_{j}e_{j}$
for $1\le j\le d$. For $t\in\mathbb{R}$ set,
\[
A_{\chi}^{t}:=\mathrm{diag}\left(2^{t\chi_{1}},...,2^{t\chi_{d}}\right).
\]
For $r>0$ and $y,x\in\mathbb{R}^{d}$ we write $S_{r}x:=rx$ and
$T_{y}x:=y+x$.

\subsection{Entropy}

In this subsection we provide some useful properties of entropy. These
properties will be used repeatedly throughout the rest of the paper,
often without further reference.

Let $(X,\mathcal{F})$ be a measurable space. Given a probability
measure $\theta$ on $X$ and a countable partition $\mathcal{D}\subset\mathcal{F}$
of $X$, the entropy of $\theta$ with respect to $\mathcal{D}$ is
defined by
\[
H(\theta,\mathcal{D}):=-\sum_{D\in\mathcal{D}}\theta(D)\log\theta(D).
\]
If $\mathcal{E}\subset\mathcal{F}$ is another countable partition
of $X$, the conditional entropy given $\mathcal{E}$ is defined as
follows
\[
H(\theta,\mathcal{D}\mid\mathcal{E}):=\sum_{E\in\mathcal{E}}\theta(E)\cdot H(\theta_{E},\mathcal{D}),
\]
where $\theta_{E}:=\theta(E)^{-1}\theta|_{E}$ for $E\in\mathcal{E}$
with $\theta(E)>0$. For a sub-$\sigma$-algebra $\mathcal{A}$ of
$\mathcal{F}$, the conditional entropy given $\mathcal{A}$ is defined
by
\[
H(\theta,\mathcal{D}\mid\mathcal{A}):=\int-\sum_{D\in\mathcal{D}}\theta(D\mid\mathcal{A})\log\theta(D\mid\mathcal{A})\:d\theta,
\]
where $\theta(D\mid\mathcal{A})$ is the conditional probability of
$D$ given $\mathcal{A}$.

In what follows all entropies are assumed to be finite. Note that
we always have the following upper bound,
\begin{equation}
H(\theta,\mathcal{D})\le\log\#\{D\in\mathcal{D}\::\:\theta(D)>0\}.\label{eq:card ub for ent}
\end{equation}

Given another countable partition $\mathcal{Q}\subset\mathcal{F}$
of $X$, conditional entropy satisfies the formula
\begin{equation}
H(\theta,\mathcal{D}\vee\mathcal{E}\mid\mathcal{Q})=H(\theta,\mathcal{E}\mid\mathcal{Q})+H(\theta,\mathcal{D}\mid\mathcal{Q}\vee\mathcal{E}),\label{eq:extended cond ent form}
\end{equation}
where
\[
\mathcal{D}\vee\mathcal{E}:=\{D\cap E\::\:D\in\mathcal{D}\text{ and }E\in\mathcal{E}\}.
\]
By taking $\mathcal{Q}$ to be the trivial partition,
\begin{equation}
H(\theta,\mathcal{D}\mid\mathcal{E})=H(\theta,\mathcal{D}\vee\mathcal{E})-H(\theta,\mathcal{E}).\label{eq:cond ent form}
\end{equation}

If $(Y,\mathcal{G})$ is another measurable space, $\mathcal{C}\subset\mathcal{G}$
is a countable partition of $Y$, and $f:X\rightarrow Y$ is measurable,
we have
\[
H(\theta,f^{-1}\mathcal{C})=H(f\theta,\mathcal{C}),
\]
where $f^{-1}\mathcal{C}:=\left\{ f^{-1}(C)\::\:C\in\mathcal{C}\right\} $.

The conditional entropy $H(\theta,\mathcal{D}\mid\mathcal{E})$ is
increasing in the $\mathcal{D}$-argument and decreasing in the $\mathcal{E}$-argument.
More precisely, if $\mathcal{D}',\mathcal{E}'\subset\mathcal{F}$
are countable partitions refining $\mathcal{D},\mathcal{E}$ respectively,
we have
\[
H(\theta,\mathcal{D}\mid\mathcal{E}')\le H(\theta,\mathcal{D}\mid\mathcal{E})\le H(\theta,\mathcal{D}'\mid\mathcal{E}).
\]

The entropy and conditional entropy functions are concave and almost
convex in the measure argument. That is, given probability measures
$\theta_{1},...,\theta_{k}$ on $X$ and a probability vector $q=(q_{i})_{i=1}^{k}$
so that $\theta=\sum_{i=1}^{k}q_{i}\theta_{i}$, we have
\begin{equation}
\sum_{i=1}^{k}q_{i}H(\theta_{i},\mathcal{D})\le H(\theta,\mathcal{D})\le\sum_{i=1}^{k}q_{i}H(\theta_{i},\mathcal{D})+H(q),\label{eq:conc =000026 almo conv of ent}
\end{equation}
where $H(q):=-\sum_{i=1}^{k}q_{i}\log q_{i}$ is the entropy of $q$.
These inequalities remain true with $H(\cdot,\mathcal{D}\mid\mathcal{E})$
in place of $H(\cdot,\mathcal{D})$.

Given $C\ge1$, we say that $\mathcal{D}$ and $\mathcal{E}$ are
$C$-commensurable if for each $D\in\mathcal{D}$ and $E\in\mathcal{E}$,
\[
\#\{E'\in\mathcal{E}\::\:E'\cap D\ne\emptyset\}\le C\text{ and }\#\{D'\in\mathcal{D}\::\:D'\cap E\ne\emptyset\}\le C.
\]
From (\ref{eq:card ub for ent}) and (\ref{eq:cond ent form}), it
follows easily that 
\[
\left|H(\theta,\mathcal{D})-H(\theta,\mathcal{E})\right|\le2\log C,
\]
whenever $\mathcal{D}$ and $\mathcal{E}$ are $C$-commensurable.

\subsection{Non-conformal partitions}

For $m\ge1$ and $n\in\mathbb{Z}$ let $\mathcal{D}_{n}^{m}$ denote
the level-$n$ dyadic partition of $\mathbb{R}^{m}$. That is,
\[
\mathcal{D}_{n}^{m}:=\left\{ \Big[\frac{k_{1}}{2^{n}},\frac{k_{1}+1}{2^{n}}\Big)\times...\times\Big[\frac{k_{m}}{2^{n}},\frac{k_{m}+1}{2^{n}}\Big)\::\:k_{1},...,k_{m}\in\mathbb{Z}\right\} .
\]
We sometimes omit the superscript $m$ from the notation when it is
clear from the context. Given $t\in\mathbb{R}$ we write $\mathcal{D}_{t}^{m}$
in place of $\mathcal{D}_{\left\lfloor t\right\rfloor }^{m}$, where
$\left\lfloor t\right\rfloor $ is the integral part of $t$.

For $n\in\mathbb{Z}$ set,
\[
\mathcal{E}_{n}:=\left\{ D_{1}\times...\times D_{d}\::\:D_{j}\in\mathcal{D}_{\chi_{j}n}^{1}\text{ for }1\le j\le d\right\} .
\]
We refer to $\mathcal{E}_{n}$ as the level-$n$ non-conformal partition
of $\mathbb{R}^{d}$ associated to $\chi_{1},...,\chi_{d}$.

It is easy to verify that,
\begin{equation}
\pi_{J}^{-1}\mathcal{E}_{n}\text{ and }T_{y}^{-1}\pi_{J}^{-1}\mathcal{E}_{n}\text{ are }O(1)\text{-commensurable for }J\subset[d]\text{ and }y\in\mathbb{R}^{d}.\label{eq:comens part after trans}
\end{equation}
It is also easy to check that,
\begin{equation}
A_{\chi}^{-t}\pi_{J}^{-1}\mathcal{E}_{n}\text{ and }\pi_{J}^{-1}\mathcal{E}_{n+t}\text{ are }O(1)\text{-commensurable for }J\subset[d]\text{ and }t\in\mathbb{R}.\label{eq:comens part after scaling}
\end{equation}

Let $X$ be a nonempty set and let $f,g:X\rightarrow\mathbb{R}^{d}$
and $C>1$ be with,
\[
|\pi_{j}\left(f(x)-g(x)\right)|\le C2^{-n\chi_{j}}\text{ for }1\le j\le d\text{ and }x\in X.
\]
In this case,
\begin{equation}
f^{-1}\pi_{J}^{-1}\mathcal{E}_{n}\text{ and }g^{-1}\pi_{J}^{-1}\mathcal{E}_{n}\text{ are }O\left(C^{d}\right)\text{-commensurable for }J\subset[d].\label{eq:comens part under inv img}
\end{equation}

\subsection{\label{subsec:Component-measures}Component measures}

For a set $X$, a partition $\mathcal{D}$ of $X$, and $x\in X$,
we denote by $\mathcal{D}(x)$ the unique $D\in\mathcal{D}$ containing
$x$.

Given $\theta\in\mathcal{M}(\mathbb{R}^{d})$ and $n\ge0$, let $\theta_{x,n}$
be a measure valued random element such that $\theta_{x,n}=\theta_{\mathcal{E}_{n}(x)}$
with probability $\theta(\mathcal{E}_{n}(x))$ for each $x\in\mathbb{R}^{d}$.
Thus, for an event $\mathcal{U}\subset\mathcal{M}(\mathbb{R}^{d})$,
\[
\mathbb{P}\left(\theta_{x,n}\in\mathcal{U}\right)=\theta\left\{ x\in\mathbb{R}^{d}\::\:\theta_{\mathcal{E}_{n}(x)}\in\mathcal{U}\right\} .
\]
We call $\theta_{x,n}$ an $n$th level (non-conformal) component
of $\theta$. For a given $x\in\mathbb{R}^{d}$ with $\theta(\mathcal{E}_{n}(x))>0$,
we often write $\theta_{x,n}$ in place of $\theta_{\mathcal{E}_{n}(x)}$
even when no randomness is involved.

Sometimes $n$ is chosen randomly as well, usually uniformly in some
range. For example, for integers $n_{2}\geq n_{1}$
\[
\mathbb{P}_{n_{1}\leq i\leq n_{2}}\left(\theta_{x,i}\in\mathcal{U}\right):=\frac{1}{n_{2}-n_{1}+1}\sum_{i=n_{1}}^{n_{2}}\mathbb{P}(\theta_{x,i}\in\mathcal{U}).
\]

We denote by $\mathbb{E}$ and $\mathbb{E}_{n_{1}\leq i\leq n_{2}}$
the expected value with respect to the probabilities $\mathbb{P}$
and $\mathbb{P}_{n_{1}\leq i\leq n_{2}}$. Thus, for a bounded measurable
$f:\mathcal{M}(\mathbb{R}^{d})\rightarrow\mathbb{R}$,
\[
\mathbb{E}_{i=n}\left(f(\theta_{x,i})\right)=\int f(\theta_{\mathcal{E}_{n}(x)})\:d\theta(x).
\]
Note that for $k\ge1$,
\[
H(\theta,\mathcal{D}_{n+k}|\mathcal{D}_{n})=\mathbb{E}_{i=n}(H(\theta_{x,i},\mathcal{D}_{i+k})).
\]

We finish this subsection with the following useful lemma. The proof
is identical to the proof of \cite[Lemma 3.4]{Ho1} and is therefore
omitted.
\begin{lem}
\label{lem:ent=00003Davg of loc ent}Let $R\ge1$ and $\theta\in\mathcal{M}_{\mathrm{c}}(\mathbb{R}^{d})$
be with $\mathrm{diam}(\mathrm{supp}(\theta))\le R$. Then for every
$n\ge m\ge1$,
\[
\frac{1}{n}H\left(\theta,\mathcal{E}_{n}\right)=\mathbb{E}_{1\le i\le n}\left(\frac{1}{m}H\left(\theta_{x,i},\mathcal{E}_{i+m}\right)\right)+O\left(\frac{m+\log R}{n}\right).
\]
\end{lem}

\section{\label{sec:Entropy-of-repeated-conv}Entropy of repeated self-convolutions}

The purpose of this section is to prove Proposition \ref{prop:full ent of slices}.
We start with some necessary preparations.

\subsection{The Berry-Esseen theorem}

Given $\theta\in\mathcal{M}(\mathbb{R})$ with a finite second moment,
we denote its mean by $\left\langle \theta\right\rangle $ and its
variance by $\mathrm{Var}(\theta)$. That is,
\[
\left\langle \theta\right\rangle :=\int x\:d\theta(x)\text{ and }\mathrm{Var}(\theta):=\int\left(x-\left\langle \theta\right\rangle \right)^{2}\:d\theta(x).
\]
For $s>0$ we write $\gamma_{s}\in\mathcal{M}(\mathbb{R})$ for the
Gaussian measure with mean $0$ and variance $s$. That is,
\[
\gamma_{s}(B)=\int_{B}\frac{1}{\sqrt{2\pi s}}\exp\left(-x^{2}/(2s)\right)\:dx\text{ for }B\subset\mathbb{R}\text{ Borel}.
\]
For $\theta\in\mathcal{M}(\mathbb{R})$ we denote its cumulative distribution
function by $F_{\theta}$. That is, $F_{\theta}(t)=\theta(-\infty,t]$
for $t\in\mathbb{R}$. We shall need the following version from \cite{Esseen}
of the Berry-Esseen theorem.
\begin{thm}
\label{thm:Berry-Esseen}There exists a global constant $C>1$ so
that the following holds. Let $\theta_{1},...,\theta_{k}\in\mathcal{M}(\mathbb{R})$
be given. For $1\le i\le k$ set $\rho_{i}:=\int|t|^{3}\:d\theta_{i}(t)$,
and assume that $\rho_{1},...,\rho_{k}<\infty$. Set $\sigma:=\theta_{1}*...*\theta_{k}$,
suppose that $\left\langle \sigma\right\rangle =0$, and set $s:=\mathrm{Var}(\sigma)$.
Then,
\[
\left|F_{\sigma}(t)-F_{\gamma_{s}}(t)\right|\le\frac{C}{\mathrm{Var}(\sigma)^{3/2}}\sum_{i=1}^{k}\rho_{i}\text{ for all }t\in\mathbb{R}.
\]
\end{thm}

\subsection{A key lemma}

The purpose of this subsection it to prove the following lemma. For
$1\le j\le d$ and $\sigma\in\mathcal{M}_{\mathrm{c}}(e_{j}\mathbb{R})$
we write $\mathrm{Var}(\sigma)$ in place of $\mathrm{Var}(\sigma')$,
where $\sigma'\in\mathcal{M}_{\mathrm{c}}(\mathbb{R})$ is the push-forward
of $\sigma$ via the map sending $te_{j}$ to $t$.
\begin{lem}
\label{lem:ent of slices of O(1) measures}Let $0<\epsilon<1$ and
$m,l,k\in\mathbb{Z}_{>0}$ be with $\epsilon^{-1}\ll m\ll l\ll k$,
and let $\theta_{1},...,\theta_{k}\in\mathcal{M}_{\mathrm{c}}(\mathbb{R}^{d})$
be with $\mathrm{diam}(\mathrm{supp}(\theta_{i}))\le\epsilon^{-1}$
for $1\le i\le k$. Set $\sigma:=\theta_{1}*...*\theta_{k}$, and
suppose that there exists $1\le j\le d$ so that $\mathrm{Var}(\pi_{j}\sigma)\ge\epsilon k$
and $\mathrm{Var}(\pi_{j'}\sigma)\le\epsilon^{-1}$ for $1\le j'<j$.
Then for $a:=\left\lfloor \frac{1}{2\chi_{j}}\log k\right\rfloor $,
\[
\frac{1}{m}H\left(\sigma,\mathcal{E}_{l-a+m}\mid\mathcal{E}_{l-a}\vee\pi_{[d]\setminus\{j\}}^{-1}\mathcal{E}_{l-a+m}\right)>\chi_{j}-\epsilon.
\]
\end{lem}

Write $d_{\mathrm{L}}$ for the Lévy metric on $\mathcal{M}(\mathbb{R})$.
That is, for $\theta,\sigma\in\mathcal{M}(\mathbb{R})$
\[
d_{\mathrm{L}}(\theta,\sigma):=\inf\left\{ \epsilon>0\::\:F_{\theta}(t-\epsilon)-\epsilon\le F_{\sigma}(t)\le F_{\theta}(t+\epsilon)+\epsilon\text{ for all }t\in\mathbb{R}\right\} .
\]
It is well known that the topology induced by $d_{\mathrm{L}}$ is
equal to the topology of weak convergence.
\begin{lem}
\label{lem:close to normal --> full ent}For every $0<\epsilon<1$,
$m\ge1$ and $l\ge L(\epsilon,m)\ge1$, there exists $\delta=\delta(\epsilon,m,l)>0$
such that
\[
\frac{1}{m}H\left(\theta,\mathcal{D}_{l+m}^{1}\mid\mathcal{D}_{l}^{1}\right)>1-\epsilon
\]
for all $\theta\in\mathcal{M}(\mathbb{R})$ with $d_{\mathrm{L}}(\theta,\gamma_{s})\le\delta$
for some $\epsilon\le s\le\epsilon^{-1}$.
\end{lem}

\begin{proof}
Let $0<\epsilon<1$ and $m\ge1$ be given. For $s>0$ write $f_{s}$
for the density of $\gamma_{s}$. Since the map taking $(s,x)\in\mathbb{R}_{>0}\times\mathbb{R}$
to $f_{s}(x)$ is smooth, it is easy to verify that there exists $L\ge1$
so that
\begin{equation}
\frac{1}{m}H\left(\gamma_{s},\mathcal{D}_{l+m}\mid\mathcal{D}_{l}\right)>1-\epsilon/2\text{ for all }\epsilon\le s\le\epsilon^{-1}\text{ and }l\ge L.\label{eq:all s all l}
\end{equation}

Let $l\ge L$ be given. Set $\Gamma:=\{\gamma_{s}\}_{\epsilon\le s\le\epsilon^{-1}}$,
and write $\tau_{\mathrm{w}}$ for the topology of weak convergence
on $\mathcal{M}(\mathbb{R})$. From (\ref{eq:all s all l}) it follows
that there exists $V\in\tau_{\mathrm{w}}$ with $\Gamma\subset V$
so that,
\begin{equation}
\frac{1}{m}H\left(\theta,\mathcal{D}_{l+m}\mid\mathcal{D}_{l}\right)>1-\epsilon\text{ for all }\theta\in V.\label{eq:all theta in V}
\end{equation}
Since $\Gamma$ is compact with respect to $\tau_{\mathrm{w}}$ and
since $d_{\mathrm{L}}$ induces $\tau_{\mathrm{w}}$, there exists
$\delta>0$ so that $\theta\in V$ for all $\theta\in\mathcal{M}(\mathbb{R})$
with $d_{\mathrm{L}}(\theta,\gamma)\le\delta$ for some $\gamma\in\Gamma$.
This together with (\ref{eq:all theta in V}) completes the proof
of the lemma.
\end{proof}
\begin{proof}[Proof of Lemma \ref{lem:ent of slices of O(1) measures}]
Let $0<\epsilon<1$, $m,l,k\in\mathbb{Z}_{>0}$, $\theta_{1},...,\theta_{k},\sigma\in\mathcal{M}_{\mathrm{c}}(\mathbb{R}^{d})$
and $1\le j\le d$ be as in the statement of the lemma. Write $a:=\left\lfloor \frac{1}{2\chi_{j}}\log k\right\rfloor $.
For $1\le j'\le d$ and $x\in\mathbb{R}^{d}$ set $\tilde{\pi}_{j'}(x):=\left\langle e_{j'},x\right\rangle $.
Because of (\ref{eq:comens part after trans}), we may assume that
$\left\langle \tilde{\pi}_{j'}\sigma\right\rangle =0$ for $1\le j'\le d$
and $\mathrm{supp}(\theta_{i})\subset\left[-\epsilon^{-1},\epsilon^{-1}\right]^{d}$
for $1\le i\le k$.

For $1\le i\le k$ and $s=2,3$,
\[
\int|t|^{s}\:d\tilde{\pi}_{j}A_{\chi}^{-a}\theta_{i}(t)=2^{-a\chi_{j}s}\int|t|^{s}\:d\tilde{\pi}_{j}\theta_{i}(t)=O\left(\epsilon^{-s}k^{-s/2}\right).
\]
Thus,
\[
\mathrm{Var}\left(\tilde{\pi}_{j}A_{\chi}^{-a}\sigma\right)=\sum_{i=1}^{k}\mathrm{Var}\left(\tilde{\pi}_{j}A_{\chi}^{-a}\theta_{i}\right)=O\left(\epsilon^{-2}\right).
\]
Moreover,
\[
\mathrm{Var}\left(\tilde{\pi}_{j}A_{\chi}^{-a}\sigma\right)=2^{-2a\chi_{j}}\mathrm{Var}\left(\tilde{\pi}_{j}\sigma\right)\ge\epsilon.
\]
Hence,
\[
\frac{\sum_{i=1}^{k}\int|t|^{3}\:d\tilde{\pi}_{j}A_{\chi}^{-a}\theta_{i}(t)}{\mathrm{Var}(\tilde{\pi}_{j}A_{\chi}^{-a}\sigma)^{3/2}}=O\left(\epsilon^{-9/2}k^{-1/2}\right).
\]
From all of this, by Theorem \ref{thm:Berry-Esseen}, and from Lemma
\ref{lem:close to normal --> full ent}, it follows that
\[
\frac{1}{m}H\left(\tilde{\pi}_{j}A_{\chi}^{-a}\sigma,\mathcal{D}_{\chi_{j}(l+m)}^{1}\mid\mathcal{D}_{\chi_{j}l}^{1}\right)>\chi_{j}-\epsilon.
\]
Setting $\mathcal{C}:=\mathcal{E}_{l}\vee\pi_{[d]\setminus\{j\}}^{-1}\mathcal{E}_{l+m}$,
the last inequality implies that
\begin{equation}
\frac{1}{m}H\left(\pi_{j}A_{\chi}^{-a}\sigma,\mathcal{E}_{l+m}\mid\mathcal{C}\right)>\chi_{j}-\epsilon.\label{eq:proj ent is large}
\end{equation}

For $j'\in[d]\setminus\{j\}$ set
\[
S_{j'}:=\left\{ x\in\mathbb{R}^{d}\::\:\left|\pi_{j'}A_{\chi}^{-a}x\right|\le2^{-\chi_{d}(l+m)}\right\} ,
\]
and write $S:=\cap_{j'\in[d]\setminus\{j\}}S_{j'}$. For $x\in S$,
\[
\left|A_{\chi}^{-a}x-\pi_{j}A_{\chi}^{-a}x\right|=O\left(2^{-\chi_{d}(l+m)}\right).
\]
Hence from (\ref{eq:comens part under inv img}),
\begin{equation}
H\left(A_{\chi}^{-a}\sigma_{S},\mathcal{E}_{l+m}\mid\mathcal{C}\right)=H\left(\pi_{j}A_{\chi}^{-a}\sigma_{S},\mathcal{E}_{l+m}\mid\mathcal{C}\right)+O(1).\label{eq:ent close to proj ent}
\end{equation}

For $j<j'\le d$,
\[
\mathrm{Var}(\tilde{\pi}_{j'}A_{\chi}^{-a}\sigma)=2^{-2a\chi_{j'}}\sum_{i=1}^{k}\mathrm{Var}(\tilde{\pi}_{j'}\theta_{i})=O\left(\epsilon^{-2}k^{1-\chi_{j'}/\chi_{j}}\right).
\]
Additionally, by assumption, for $1\le j'<j$
\[
\mathrm{Var}(\tilde{\pi}_{j'}A_{\chi}^{-a}\sigma)\le\epsilon^{-1}2^{-2a\chi_{j'}}=O\left(\epsilon^{-1}k^{-\chi_{j'}/\chi_{j}}\right).
\]
Recall from Section \ref{subsec:The-setup} that $\chi_{1}<...<\chi_{d}$.
Thus, there exists $\eta>0$, depending only on $\chi_{1},...,\chi_{d}$,
so that
\[
\mathrm{Var}(\tilde{\pi}_{j'}A_{\chi}^{-a}\sigma)=O\left(\epsilon^{-2}k^{-\eta}\right)\text{ for }j'\in[d]\setminus\{j\}.
\]
From this, since $\left\langle \tilde{\pi}_{j'}\sigma\right\rangle =0$
for $1\le j'\le d$, and by Chebyshev's inequality,
\begin{multline}
\sigma\left(S^{c}\right)\le\sum_{j'\in[d]\setminus\{j\}}\sigma\left(S_{j'}^{c}\right)\le\sum_{j'\in[d]\setminus\{j\}}2^{2\chi_{d}(l+m)}\mathrm{Var}\left(\tilde{\pi}_{j'}A_{\chi}^{-a}\sigma\right)\\
=O\left(2^{2\chi_{d}(l+m)}\epsilon^{-2}k^{-\eta}\right).\label{eq:by Chebyshev}
\end{multline}

Since $\mathrm{supp}(\theta_{i})\subset\left[-\epsilon^{-1},\epsilon^{-1}\right]^{d}$
for $1\le i\le k$ and since $A_{\chi}^{-a}$ is contracting, we get
that $\mathrm{supp}(\pi_{j}A_{\chi}^{-a}\sigma)\subset\left[-k\epsilon^{-1},k\epsilon^{-1}\right]^{d}$.
Hence,
\[
H\left(\pi_{j}A_{\chi}^{-a}\sigma_{S^{c}},\mathcal{E}_{l+m}\right)=O\left(l+m+\log\left(k\epsilon^{-1}\right)\right).
\]
From this, from (\ref{eq:by Chebyshev}), and since $\epsilon^{-1},m,l\ll k$,
\[
\sigma\left(S^{c}\right)\frac{1}{m}H\left(\pi_{j}A_{\chi}^{-a}\sigma_{S^{c}},\mathcal{E}_{l+m}\mid\mathcal{C}\right)\le\epsilon.
\]
Thus, from (\ref{eq:proj ent is large}) and (\ref{eq:ent close to proj ent}),
and by the concavity and almost convexity of conditional entropy,
\begin{eqnarray*}
\frac{1}{m}H\left(A_{\chi}^{-a}\sigma,\mathcal{E}_{l+m}\mid\mathcal{C}\right) & \ge & \frac{\sigma(S)}{m}H\left(A_{\chi}^{-a}\sigma_{S},\mathcal{E}_{l+m}\mid\mathcal{C}\right)\\
 & \ge & \frac{\sigma(S)}{m}H\left(\pi_{j}A_{\chi}^{-a}\sigma_{S},\mathcal{E}_{l+m}\mid\mathcal{C}\right)-O\left(\frac{1}{m}\right)\\
 & + & \frac{\sigma\left(S^{c}\right)}{m}H\left(\pi_{j}A_{\chi}^{-a}\sigma_{S^{c}},\mathcal{E}_{l+m}\mid\mathcal{C}\right)-\epsilon\\
 & \ge & \frac{1}{m}H\left(\pi_{j}A_{\chi}^{-a}\sigma,\mathcal{E}_{l+m}\mid\mathcal{C}\right)-2\epsilon>\chi_{j}-3\epsilon.
\end{eqnarray*}
This together with (\ref{eq:comens part after scaling}) completes
the proof of the lemma.
\end{proof}

\subsection{Proof of the proposition}

The simple proof of the following lemma is almost identical to the
proof of \cite[Lemma 4.4]{Ho1} and is therefore omitted.
\begin{lem}
\label{lem:small Var --> small ent}For every $0<\epsilon<1$ and
$m\ge M(\epsilon)\ge1$ there exists $\delta=\delta(\epsilon,m)>0$
so that the following holds. Let $\theta\in\mathcal{M}_{\mathrm{c}}(\mathbb{R}^{d})$
be with $\mathrm{diam}(\mathrm{supp}(\theta))\le\epsilon^{-1}$ and
$\mathrm{Var}(\pi_{j}\theta)\le\delta$ for each $1\le j\le d$. Then
$\frac{1}{m}H\left(\theta,\mathcal{E}_{m}\right)<\epsilon$.
\end{lem}

Recall from Section \ref{subsec:Basic-notations} that for $n\ge1$
we denote by $\lambda_{n}$ the normalized counting measure on $[n]$.
\begin{lem}
\label{lem:var of proj of comp > delta}For every $0<\epsilon<1$
there exists $\delta=\delta(\epsilon)>0$ so that the following holds.
Let $n\ge N(\epsilon)\ge1$ and $\theta\in\mathcal{M}_{\mathrm{c}}(\mathbb{R}^{d})$
be with $\mathrm{diam}(\mathrm{supp}(\theta))\le\epsilon^{-1}$ and
$\frac{1}{n}H(\theta,\mathcal{E}_{n})>\epsilon$. Then $\lambda_{n}(\mathrm{B})>\delta$,
where $\mathrm{B}$ is the set of all integers $1\le b\le n$ so that
\[
\mathbb{P}_{i=b}\left\{ \mathrm{Var}(\pi_{j}A_{\chi}^{i}\theta_{x,i})>\delta\text{ for some }1\le j\le d\right\} >\delta.
\]
\end{lem}

\begin{proof}
Let $0<\epsilon,\eta,\delta<1$ and $m,n\in\mathbb{Z}_{>0}$ be with
\[
\epsilon^{-1}\ll\eta^{-1}\ll m\ll\delta^{-1}\ll n,
\]
and let $\theta\in\mathcal{M}_{\mathrm{c}}(\mathbb{R}^{d})$ be with
$\mathrm{diam}(\mathrm{supp}(\theta))\le\epsilon^{-1}$ and $\frac{1}{n}H(\theta,\mathcal{E}_{n})>\epsilon$.
By Lemma \ref{lem:ent=00003Davg of loc ent} and since $\frac{1}{n}H(\theta,\mathcal{E}_{n})>\epsilon$,
\[
\mathbb{E}_{1\le i\le n}\left(\frac{1}{m}H\left(\theta_{x,i},\mathcal{E}_{i+m}\right)\right)>\epsilon/2.
\]
Thus, by (\ref{eq:comens part after scaling}),
\[
\mathbb{E}_{1\le i\le n}\left(\frac{1}{m}H\left(A_{\chi}^{i}\theta_{x,i},\mathcal{E}_{m}\right)\right)>\epsilon/3.
\]
Note that,
\[
\frac{1}{m}H\left(A_{\chi}^{i}\theta_{x,i},\mathcal{E}_{m}\right)=O(1)\text{ for }1\le i\le n\text{ and }\theta\text{-a.e. }x.
\]
Hence,
\[
\mathbb{P}_{1\le i\le n}\left(\frac{1}{m}H\left(A_{\chi}^{i}\theta_{x,i},\mathcal{E}_{m}\right)\ge\eta\right)>\eta.
\]
This together with Lemma \ref{lem:small Var --> small ent} completes
the proof of the lemma.
\end{proof}
We are now ready to prove Proposition \ref{prop:full ent of slices},
whose statement we first recall.
\begin{prop*}
For every $0<\epsilon<1$ there exists $\delta>0$ so that the following
holds. Let $m_{1},...,m_{d},k_{1},...,k_{d},n\in\mathbb{Z}_{>0}$
be such that $\epsilon^{-1}\ll m_{d}$, $k_{1}\ll n$, $m_{j}\ll k_{j}$
for $1\le j\le d$, and $k_{j+1}\ll m_{j}$ for $1\le j<d$. Let $\theta\in\mathcal{M}_{\mathrm{c}}(\mathbb{R}^{d})$
be with $\mathrm{diam}(\mathrm{supp}(\theta))\le\epsilon^{-1}$ and
$\frac{1}{n}H(\theta,\mathcal{E}_{n})>\epsilon$. For $1\le j\le d$
write $\mathrm{Q}_{j}$ for the set of all integers $1\le q\le n$
so that,
\[
\frac{1}{m_{j}}H\left(\theta^{*k_{j}},\mathcal{E}_{q+m_{j}}\mid\mathcal{E}_{q}\vee\pi_{[d]\setminus\{j\}}^{-1}\mathcal{E}_{q+m_{j}}\right)>\chi_{j}-\epsilon.
\]
Then there exists $1\le j\le d$ such that $\lambda_{n}\left(\mathrm{Q}_{j}\right)>\delta$.
\end{prop*}
\begin{proof}
Let $0<\epsilon<1$ be given, and let $m_{1},...,m_{d},k_{1},...,k_{d},n\in\mathbb{Z}_{>0}$
and $\theta\in\mathcal{M}_{\mathrm{c}}(\mathbb{R}^{d})$ be as in
the statement of the proposition. 

Let $0<\delta<1$ be with $\epsilon^{-1}\ll\delta^{-1}\ll m_{d}$.
By Lemma \ref{lem:var of proj of comp > delta} we may assume that
$\lambda_{n}(\mathrm{B})>\delta$, where $\mathrm{B}$ is the set
of all $1\le b\le n$ so that
\[
\mathbb{P}_{i=b}\left\{ \mathrm{Var}(\pi_{j}A_{\chi}^{i}\theta_{x,i})>\delta\text{ for some }1\le j\le d\right\} >\delta.
\]
Let $0<\eta_{d}<1$ be with $\delta^{-1}\ll\eta_{d}^{-1}\ll m_{d}$,
and for $1\le j<d$ set $\eta_{j}:=k_{j+1}^{-1}$. For $1\le j\le d$
let $\mathrm{B}_{j}$ be the set of all $1\le b\le n$ so that,
\[
\mathbb{P}_{i=b}\left\{ \mathrm{Var}\left(\pi_{j}A_{\chi}^{i}\theta_{x,i}\right)>\eta_{j}\text{ and }\mathrm{Var}\left(\pi_{j'}A_{\chi}^{i}\theta_{x,i}\right)\le\eta_{j'}\text{ for }1\le j'<j\right\} >\delta/d.
\]
It clearly holds that $\mathrm{B}\subset\cup_{j=1}^{d}\mathrm{B}_{j}$.
Thus, from $\lambda_{n}(\mathrm{B})>\delta$ it follows that $\lambda_{n}(\mathrm{B}_{j})>\delta/d$
for some $1\le j\le d$. Fix such a $j$ until the end of the proof.

Let $l\in\mathbb{Z}_{>0}$ be with $m_{j}\ll l\ll k_{j}$. Note that,
\[
\epsilon^{-1}\ll\delta^{-1}\ll\eta_{j}^{-1}\ll m_{j}\ll l\ll k_{j}\ll n\text{ and }\eta_{j'}\le k_{j}^{-1}\text{ for }1\le j'<j.
\]
Let $b\in\mathrm{B}_{j}$ be given, and set
\[
Y:=\left\{ x\in\mathbb{R}^{d}\::\:\mathrm{Var}\left(\pi_{j}A_{\chi}^{b}\theta_{x,b}\right)>\eta_{j}\text{ and }\mathrm{Var}\left(\pi_{j'}A_{\chi}^{b}\theta_{x,b}\right)\le\eta_{j'}\text{ for }1\le j'<j\right\} .
\]
Also, set $k:=\left\lfloor \frac{\delta k_{j}}{2d}\right\rfloor $
and
\[
Z:=\left\{ (x_{1},...,x_{k_{j}})\in(\mathbb{R}^{d})^{k_{j}}\::\:\#\left\{ 1\le s\le k_{j}\::\:x_{s}\in Y\right\} \ge k\right\} .
\]
Since $b\in\mathrm{B}_{j}$ and $\delta^{-1}\ll k_{j}$, it follows
by the weak law of large numbers that $\theta^{\times k_{j}}(Z)>1-\delta$.

Let $(x_{1},...,x_{k_{j}})\in Z$ be given. By the definition of $Z$,
there exist integers $1\le s_{1}<...<s_{k}\le k_{j}$ so that $x_{s_{i}}\in Y$
for $1\le i\le k$. Note that,
\[
\mathrm{diam}\left(\mathrm{supp}\left(A_{\chi}^{b}\theta_{x_{s_{i}},b}\right)\right)=O(1)\text{ for }1\le i\le k.
\]
Set,
\[
\sigma:=A_{\chi}^{b}\theta_{x_{s_{1}},b}*...*A_{\chi}^{b}\theta_{x_{s_{k}},b}.
\]
We have
\[
\mathrm{Var}(\pi_{j}\sigma)=\sum_{i=1}^{k}\mathrm{Var}\left(\pi_{j}A_{\chi}^{b}\theta_{x_{s_{i}},b}\right)\ge k\eta_{j},
\]
and for each $1\le j'<j$
\[
\mathrm{Var}(\pi_{j'}\sigma)=\sum_{i=1}^{k}\mathrm{Var}\left(\pi_{j'}A_{\chi}^{b}\theta_{x_{s_{i}},b}\right)\le k\eta_{j'}\le1.
\]
Moreover, from $l\ll k_{j}$ and $\delta^{-1}\ll k_{j}$ we get $l\ll k$.
From these facts, since $\delta^{-1},\eta_{j}^{-1}\ll m_{j}\ll l$,
and by Lemma \ref{lem:ent of slices of O(1) measures}, it follows
that for $a:=\left\lfloor \frac{1}{2\chi_{j}}\log k\right\rfloor $
we have
\begin{equation}
\frac{1}{m_{j}}H\left(\sigma,\mathcal{E}_{l-a+m_{j}}\mid\mathcal{E}_{l-a}\vee\pi_{[d]\setminus\{j\}}^{-1}\mathcal{E}_{l-a+m_{j}}\right)>\chi_{j}-\delta.\label{eq:lb on ent of sig in prop}
\end{equation}

For $s\in\mathbb{Z}$ set $\mathcal{C}_{s}:=\mathcal{E}_{s+l-a}\vee\pi_{[d]\setminus\{j\}}^{-1}\mathcal{E}_{s+l-a+m_{j}}$.
From (\ref{eq:lb on ent of sig in prop}), by (\ref{eq:comens part after trans}),
and by the concavity of conditional entropy,
\[
\frac{1}{m_{j}}H\left(*_{s=1}^{k_{j}}A_{\chi}^{b}\theta_{x_{s},b},\mathcal{E}_{l-a+m_{j}}\mid\mathcal{C}_{0}\right)>\chi_{j}-2\delta,
\]
where $*_{s=1}^{k_{j}}A_{\chi}^{b}\theta_{x_{s},b}$ denotes the convolution
of $A_{\chi}^{b}\theta_{x_{1},b},A_{\chi}^{b}\theta_{x_{2},b},...,A_{\chi}^{b}\theta_{x_{k_{j}},b}$.
From the last inequality and by (\ref{eq:comens part after scaling}),
\begin{equation}
\frac{1}{m_{j}}H\left(*_{s=1}^{k_{j}}\theta_{x_{s},b},\mathcal{E}_{b+l-a+m_{j}}\mid\mathcal{C}_{b}\right)>\chi_{j}-3\delta\text{ for }(x_{1},...,x_{k_{j}})\in Z.\label{eq:for all x in Z}
\end{equation}

By the decomposition
\[
\theta^{*k_{j}}=\int*_{s=1}^{k_{j}}\theta_{x_{s},b}\:\;d\theta^{\times k_{j}}(x_{1},...,x_{k_{j}}),
\]
from the concavity of conditional entropy, by (\ref{eq:for all x in Z}),
and since $\theta^{\times k_{j}}(Z)>1-\delta$, it follows that for
all $b\in\mathrm{B}_{j}$
\begin{multline}
\frac{1}{m_{j}}H\left(\theta^{*k_{j}},\mathcal{E}_{b+l-a+m_{j}}\mid\mathcal{E}_{b+l-a}\vee\pi_{[d]\setminus\{j\}}^{-1}\mathcal{E}_{b+l-a+m_{j}}\right)\ge\\
\int_{Z}\frac{1}{m_{j}}H\left(*_{s=1}^{k_{j}}\theta_{x_{s},b},\mathcal{E}_{b+l-a+m_{j}}\mid\mathcal{C}_{b}\right)\:d\theta^{\times k_{j}}(x_{1},...,x_{k_{j}})\ge\chi_{j}-O(\delta).\label{eq:lb all b in B_j}
\end{multline}

Let $\mathrm{B}_{j}'$ be the set of all $1\le b\le n$ so that $b-l+a\in\mathrm{B}_{j}$.
From $l,a,\delta^{-1}\ll n$ and $\lambda_{n}(\mathrm{B}_{j})>\delta/d$,
it follows that $\lambda_{n}(\mathrm{B}_{j}')>\delta/(2d)$. Moreover,
since $\epsilon^{-1}\ll\delta^{-1}$ and by (\ref{eq:lb all b in B_j}),
\[
\frac{1}{m_{j}}H\left(\theta^{*k_{j}},\mathcal{E}_{b+m_{j}}\mid\mathcal{E}_{b}\vee\pi_{[d]\setminus\{j\}}^{-1}\mathcal{E}_{b+m_{j}}\right)>\chi_{j}-\epsilon\text{ for }b\in\mathrm{B}_{j}',
\]
which completes the proof of the proposition.
\end{proof}

\section{\label{sec:Asymptotic-entropies-of mu}Asymptotic entropies of $\mu$}

Recall from Section \ref{subsec:About-the-proof} that,
\[
\kappa:=\chi_{d}\dim\mu-\sum_{j=1}^{d-1}(\chi_{d}-\chi_{j}).
\]
The purpose of this section is to show that $\frac{1}{n}H\left(\mu,\mathcal{E}_{n}\right)$
tends to $\kappa$ when $\dim\pi_{[d-1]}\mu=d-1$, and to prove Lemmata
\ref{lem:lb on ent of comp} and \ref{lem:lb on proj of comp}.

Given $1\le j\le d$, $n\ge1$ and $i_{1}...i_{n}=u\in\Lambda^{n}$
we write $r_{u,j}:=r_{i_{1},j}\cdot...\cdot r_{i_{n},j}$ and $r_{\emptyset,j}:=1$,
where $\emptyset$ is the empty word here. From (\ref{eq:diag mat in 1-dim subgroup})
it follows that for each $u\in\Lambda^{*}$ there exists $t_{u}\in\mathbb{R}_{\ge0}$
so that,
\begin{equation}
\mathrm{diag}\left(|r_{u,1}|,...,|r_{u,d}|\right)=A_{\chi}^{-t_{u}}.\label{eq:lin part =00003D A^-t}
\end{equation}
Writing $uv$ for the concatenation of $u,v\in\Lambda^{*}$, we clearly
have $t_{uv}=t_{u}+t_{v}$. For $n\ge1$ set
\[
\mathcal{U}_{n}:=\left\{ i_{1}...i_{l}\in\Lambda^{*}\::\:t_{i_{1}...i_{l-1}}<n\le t_{i_{1}...i_{l}}\right\} ,
\]
and note that
\begin{equation}
n\le t_{u}\le n+O(1)\text{ for all }u\in\mathcal{U}_{n}.\label{eq:t_n in =00005Bn,n+C=00005D}
\end{equation}

\begin{lem}
\label{lem:lim of ent wrt non-conf part}Suppose that $\dim\pi_{[d-1]}\mu=d-1$.
Then,
\[
\underset{n}{\lim}\:\frac{1}{n}H\left(\mu,\mathcal{E}_{n}\right)=\kappa.
\]
\end{lem}

\begin{proof}
Let us first show that,
\begin{equation}
\underset{n}{\lim}\:\frac{1}{n}H\left(\mu,\mathcal{D}_{\chi_{d}n}\mid\mathcal{E}_{n}\right)=\sum_{j=1}^{d-1}\left(\chi_{d}-\chi_{j}\right).\label{eq:lim=00003DSum_j=00003D1^d-1}
\end{equation}
For each $n\ge1$ and $E\in\mathcal{E}_{n}$
\[
\log\#\left\{ D\in\mathcal{D}_{\chi_{d}n}^{d}\::\:E\cap D\ne\emptyset\right\} =\sum_{j=1}^{d-1}\left(\left\lfloor \chi_{d}n\right\rfloor -\left\lfloor \chi_{j}n\right\rfloor \right),
\]
which implies
\begin{equation}
\underset{n}{\limsup}\:\frac{1}{n}H\left(\mu,\mathcal{D}_{\chi_{d}n}\mid\mathcal{E}_{n}\right)\le\sum_{j=1}^{d-1}\left(\chi_{d}-\chi_{j}\right).\label{eq:limsup <=00003D Sum_j=00003D1^d-1}
\end{equation}

Let $0<\epsilon<1$ and let $n\ge1$ be with $\epsilon^{-1}\ll n$.
From the decomposition $\mu=\sum_{u\in\mathcal{U}_{n}}p_{u}\cdot\varphi_{u}\mu$,
by the concavity of conditional entropy, and from (\ref{eq:lin part =00003D A^-t})
and (\ref{eq:comens part after trans}),
\[
\frac{1}{n}H\left(\mu,\mathcal{D}_{\chi_{d}n}\mid\mathcal{E}_{n}\right)\ge\sum_{u\in\mathcal{U}_{n}}p_{u}\cdot\frac{1}{n}H\left(A_{\chi}^{-t_{u}}\mu,\mathcal{D}_{\chi_{d}n}\mid\mathcal{E}_{n}\right)-O(1/n).
\]
Thus, by (\ref{eq:comens part after scaling}) and (\ref{eq:t_n in =00005Bn,n+C=00005D}),
\[
\frac{1}{n}H\left(\mu,\mathcal{D}_{\chi_{d}n}\mid\mathcal{E}_{n}\right)\ge\sum_{u\in\mathcal{U}_{n}}p_{u}\cdot\frac{1}{n}H\left(\mu,A_{\chi}^{t_{u}}\mathcal{D}_{\chi_{d}n}\right)-O(1/n).
\]
From (\ref{eq:t_n in =00005Bn,n+C=00005D}) it also follows that for
each $u\in\mathcal{U}_{n}$ the partitions $A_{\chi}^{t_{u}}\mathcal{D}_{\chi_{d}n}$
and $\vee_{j=1}^{d}\pi_{j}^{-1}\mathcal{D}_{(\chi_{d}-\chi_{j})n}$
are $O(1)$-commensurable. Hence, by the last inequality,
\begin{equation}
\frac{1}{n}H\left(\mu,\mathcal{D}_{\chi_{d}n}\mid\mathcal{E}_{n}\right)\ge\frac{1}{n}H\left(\mu,\vee_{j=1}^{d-1}\pi_{j}^{-1}\mathcal{D}_{(\chi_{d}-\chi_{j})n}\right)-O(1/n).\label{eq:>=00003D ent wrt weird part}
\end{equation}

From (\ref{eq:>=00003D ent wrt weird part}) and by the conditional
entropy formula (\ref{eq:cond ent form}),
\begin{eqnarray*}
\frac{1}{n}H\left(\mu,\mathcal{D}_{\chi_{d}n}\mid\mathcal{E}_{n}\right) & \ge & \frac{1}{n}H\left(\mu,\pi_{[d-1]}^{-1}\mathcal{D}_{(\chi_{d}-\chi_{1})n}\right)\\
 & - & \frac{1}{n}H\left(\mu,\pi_{[d-1]}^{-1}\mathcal{D}_{(\chi_{d}-\chi_{1})n}\mid\vee_{j=1}^{d-1}\pi_{j}^{-1}\mathcal{D}_{(\chi_{d}-\chi_{j})n}\right)-O(1/n).
\end{eqnarray*}
Since $\dim\pi_{[d-1]}\mu=d-1$ and $\epsilon^{-1}\ll n$,
\[
\frac{1}{n}H\left(\mu,\pi_{[d-1]}^{-1}\mathcal{D}_{(\chi_{d}-\chi_{1})n}\right)\ge(\chi_{d}-\chi_{1})(d-1)-\epsilon.
\]
For each $Q\in\vee_{j=1}^{d-1}\pi_{j}^{-1}(\mathcal{D}_{(\chi_{d}-\chi_{j})n})$
\[
\log\#\left\{ R\in\pi_{[d-1]}^{-1}\mathcal{D}_{(\chi_{d}-\chi_{1})n}\::\:Q\cap R\ne\emptyset\right\} \le d+n\sum_{j=1}^{d-1}\left(\chi_{j}-\chi_{1}\right),
\]
and so
\[
\frac{1}{n}H\left(\mu,\pi_{[d-1]}^{-1}\mathcal{D}_{(\chi_{d}-\chi_{1})n}\mid\vee_{j=1}^{d-1}\pi_{j}^{-1}\mathcal{D}_{(\chi_{d}-\chi_{j})n}\right)\le\sum_{j=1}^{d-1}\left(\chi_{j}-\chi_{1}\right)+O(1/n).
\]
By combining all of this we get
\[
\frac{1}{n}H\left(\mu,\mathcal{D}_{\chi_{d}n}\mid\mathcal{E}_{n}\right)\ge\sum_{j=1}^{d-1}\left(\chi_{d}-\chi_{j}\right)-2\epsilon,
\]
which together with (\ref{eq:limsup <=00003D Sum_j=00003D1^d-1})
implies (\ref{eq:lim=00003DSum_j=00003D1^d-1}).

From (\ref{eq:lim=00003DSum_j=00003D1^d-1}), since $\frac{1}{n}H\left(\mu,\mathcal{D}_{\chi_{d}n}\right)\rightarrow\chi_{d}\dim\mu$
as $n\rightarrow\infty$, and by the definition of $\kappa$, we now
get
\[
\kappa=\underset{n}{\lim}\:\frac{1}{n}\left(H\left(\mu,\mathcal{D}_{\chi_{d}n}\right)-H\left(\mu,\mathcal{D}_{\chi_{d}n}\mid\mathcal{E}_{n}\right)\right)=\underset{n}{\lim}\:\frac{1}{n}H\left(\mu,\mathcal{E}_{n}\right),
\]
which completes the proof of the lemma.
\end{proof}
Next we prove Lemma \ref{lem:lb on proj of comp}, whose statement
we first recall.
\begin{lem*}
Let $J\subset[d]$ be with $\dim\pi_{J}\mu=|J|$. Then for all $\epsilon>0$,
$m\ge M(\epsilon)\ge1$ and $n\ge1$,
\[
\frac{1}{m}H\left(\mu,\pi_{J}^{-1}\mathcal{E}_{n+m}\mid\mathcal{E}_{n}\right)\ge\sum_{j\in J}\chi_{j}-\epsilon.
\]
\end{lem*}
\begin{proof}
Let us first show that,
\begin{equation}
\underset{k\rightarrow\infty}{\liminf}\frac{1}{k}H\left(\pi_{J}\mu,\mathcal{E}_{k}\right)\ge\sum_{j\in J}\chi_{j}.\label{eq:liminf >=00003D sum of exp}
\end{equation}
By the conditional entropy formula it follows that for each $k\ge1$,
\[
\frac{1}{k}H\left(\pi_{J}\mu,\mathcal{E}_{k}\right)=\frac{1}{k}H\left(\pi_{J}\mu,\mathcal{D}_{\chi_{d}k}\right)-\frac{1}{k}H\left(\pi_{J}\mu,\mathcal{D}_{\chi_{d}k}\mid\mathcal{E}_{k}\right).
\]
Since $\dim\pi_{J}\mu=|J|$,
\[
\underset{k\rightarrow\infty}{\lim}\frac{1}{k}H\left(\pi_{J}\mu,\mathcal{D}_{\chi_{d}k}\right)=|J|\chi_{d}.
\]
For $k\ge1$ and $E\in\pi_{J}^{-1}\mathcal{E}_{k}$
\[
\log\#\left\{ D\in\pi_{J}^{-1}\mathcal{D}_{\chi_{d}k}\::\:D\cap E\ne\emptyset\right\} =\sum_{j\in J}\left(\left\lfloor \chi_{d}k\right\rfloor -\left\lfloor \chi_{j}k\right\rfloor \right),
\]
which implies
\[
\frac{1}{k}H\left(\pi_{J}\mu,\mathcal{D}_{\chi_{d}k}\mid\mathcal{E}_{k}\right)\le\sum_{j\in J}\left(\chi_{d}-\chi_{j}\right)+O(1/k).
\]
By combining all of this we obtain (\ref{eq:liminf >=00003D sum of exp}).

Next, let $n,m\ge1$ be given. From $\mu=\sum_{u\in\mathcal{U}_{n}}p_{u}\cdot\varphi_{u}\mu$,
by the concavity of conditional entropy, and from (\ref{eq:lin part =00003D A^-t})
and (\ref{eq:comens part after trans}),
\[
\frac{1}{m}H\left(\mu,\pi_{J}^{-1}\mathcal{E}_{n+m}\mid\mathcal{E}_{n}\right)\ge\sum_{u\in\mathcal{U}_{n}}p_{u}\cdot\frac{1}{m}H\left(A_{\chi}^{-t_{u}}\mu,\pi_{J}^{-1}\mathcal{E}_{n+m}\mid\mathcal{E}_{n}\right)-O(1/m).
\]
Thus, by (\ref{eq:comens part after scaling}) and (\ref{eq:t_n in =00005Bn,n+C=00005D}),
\[
\frac{1}{m}H\left(\mu,\pi_{J}^{-1}\mathcal{E}_{n+m}\mid\mathcal{E}_{n}\right)\ge\frac{1}{m}H\left(\mu,\pi_{J}^{-1}\mathcal{E}_{m}\right)-O(1/m).
\]
This together with (\ref{eq:liminf >=00003D sum of exp}) completes
the proof of the lemma.
\end{proof}
Finally, we prove Lemma \ref{lem:lb on ent of comp}, which is the
following statement.
\begin{lem*}
Suppose that $\dim\pi_{[d-1]}\mu=d-1$. Then for all $\epsilon>0$,
$m\ge M(\epsilon)\ge1$ and $n\ge1$,
\[
\frac{1}{m}H\left(\mu,\mathcal{E}_{n+m}\mid\mathcal{E}_{n}\right)\ge\kappa-\epsilon.
\]
\end{lem*}
\begin{proof}
Let $n,m\ge1$ be given. As in the last proof,
\[
\frac{1}{m}H\left(\mu,\mathcal{E}_{n+m}\mid\mathcal{E}_{n}\right)\ge\frac{1}{m}H\left(\mu,\mathcal{E}_{m}\right)-O(1/m).
\]
This together with Lemma \ref{lem:lim of ent wrt non-conf part} completes
the proof.
\end{proof}

\section{\label{sec:Proof-of-ent-unc-res}Proof of the entropy increase result}

The purpose of this section is to prove Theorem \ref{thm:ent inc result}.

\subsection{The Kaimanovich--Vershik lemma}

We shall need the following useful lemma, which is due to Kaimanovich
and Vershik \cite{kaimanovich_and_vershik} in some form. See \cite[Lemma 4.7]{Ho1}
for a proof in the formulation given here. In what follows, we write
$H(\theta)$ for the Shannon entropy of a discrete probability measure
$\theta$.
\begin{lem}
\label{lem:Kaimanovich=002013Vershik}Let $\Gamma$ be a countable
abelian group, and let $\theta$ and $\sigma$ be probability measures
on $\Gamma$ with $H(\theta),H(\sigma)<\infty$. Then for every $k\ge1$,
\[
H\left(\theta^{*k}*\sigma\right)-H(\sigma)\le k\left(H(\theta*\sigma)-H(\sigma)\right).
\]
\end{lem}

\begin{cor}
\label{cor:from Kaimanovich=002013Vershik}Let $\theta,\sigma\in\mathcal{M}_{\mathrm{c}}(\mathbb{R}^{d})$
and $n\ge1$ be given. Then for each $k\ge1$,
\[
H\left(\theta^{*k}*\sigma,\mathcal{E}_{n}\right)-H(\sigma,\mathcal{E}_{n})\le k\left(H(\theta*\sigma,\mathcal{E}_{n})-H(\sigma,\mathcal{E}_{n})\right)+O(k).
\]
\end{cor}

\begin{proof}
The proof is similar to the proof of \cite[Proposition 4.9]{Ho1},
but we include it here for completeness. For $1\le j\le d$ set $m_{j}:=\left\lfloor \chi_{j}n\right\rfloor $
and write
\[
\Gamma:=\left\{ \left(\frac{l_{1}}{2^{m_{1}}},...,\frac{l_{d}}{2^{m_{d}}}\right)\::\:l_{1},...,l_{d}\in\mathbb{Z}\right\} ,
\]
so that $\Gamma$ is a discrete subgroup of $\mathbb{R}^{d}$. Note
that each $E\in\mathcal{E}_{n}$ contains exactly one element from
$\Gamma$. Let $\xi:\mathbb{R}^{d}\rightarrow\Gamma$ be with $\xi(y)=x$
for $x\in\Gamma$ and $y\in\mathcal{E}_{n}(x)$.

Given $k\ge0$ let $f_{k},g_{k}:\left(\mathbb{R}^{d}\right)^{k+1}\rightarrow\mathbb{R}^{d}$
be such that for $x_{0},...,x_{k}\in\mathbb{R}^{d}$,
\[
f_{k}(x_{0},...,x_{k})=x_{0}+...+x_{k}\text{ and }g_{k}(x_{0},...,x_{k})=\xi(x_{0})+...+\xi(x_{k}).
\]
Note that
\begin{equation}
H\left(f_{k}\left(\theta^{\times k}\times\sigma\right),\mathcal{E}_{n}\right)=H\left(\theta^{*k}*\sigma,\mathcal{E}_{n}\right),\label{eq:ent of im f_k}
\end{equation}
and
\begin{equation}
H\left(g_{k}\left(\theta^{\times k}\times\sigma\right),\mathcal{E}_{n}\right)=H\left((\xi\theta)^{*k}*\xi\sigma\right).\label{eq:ent of im g_k}
\end{equation}

For $x_{0},...,x_{k}\in\mathbb{R}^{d}$ and $1\le j\le d$,
\[
\left|\pi_{j}\left(f_{k}(x_{0},...,x_{k})-g_{k}(x_{0},...,x_{k})\right)\right|\le O\left((k+1)2^{-\chi_{j}n}\right).
\]
Thus by (\ref{eq:ent of im f_k}), (\ref{eq:ent of im g_k}) and (\ref{eq:comens part under inv img}),
\begin{equation}
H\left(\theta^{*k}*\sigma,\mathcal{E}_{n}\right)=H\left((\xi\theta)^{*k}*\xi\sigma\right)+O\left(1+\log(k+1)\right)\text{ for }k\ge0.\label{eq:conv close to conv of desc}
\end{equation}
Moreover, by Lemma \ref{lem:Kaimanovich=002013Vershik},
\[
H\left((\xi\theta)^{*k}*\xi\sigma\right)-H(\xi\sigma)\le k\left(H(\xi\theta*\xi\sigma)-H(\xi\sigma)\right)\text{ for }k\ge1.
\]
This together with (\ref{eq:conv close to conv of desc}) completes
the proof of the corollary.
\end{proof}

\subsection{Proof of the theorem}

Let us recall the statement of Theorem \ref{thm:ent inc result}.
\begin{thm*}
Suppose that $\dim\mu<d$ and that $\dim\pi_{J}\mu=|J|$ for each
proper subset $J$ of $[d]$. Then for every $0<\epsilon<1$ there
exists $\delta=\delta(\epsilon)>0$ so that the following holds. Let
$n\ge N(\epsilon)\ge1$ and $\theta\in\mathcal{M}_{\mathrm{c}}(\mathbb{R}^{d})$
be with $\mathrm{diam}(\mathrm{supp}(\theta))\le\epsilon^{-1}$ and
$\frac{1}{n}H(\theta,\mathcal{E}_{n})>\epsilon$. Then $\frac{1}{n}H(\theta*\mu,\mathcal{E}_{n})>\kappa+\delta$.
\end{thm*}
\begin{proof}
Let $0<\epsilon<1$, let $n\in\mathbb{Z}_{>0}$ be with $\epsilon^{-1}\ll n$,
and let $\theta\in\mathcal{M}_{\mathrm{c}}(\mathbb{R}^{d})$ be with
$\mathrm{diam}(\mathrm{supp}(\theta))\le\epsilon^{-1}$ and $\frac{1}{n}H(\theta,\mathcal{E}_{n})>\epsilon$.
Let $0<\delta_{0}<1$ and $m_{1},...,m_{d},k_{1},...,k_{d}\in\mathbb{Z}_{>0}$
be such that $\epsilon^{-1}\ll\delta_{0}^{-1}\ll m_{d}$, $k_{1}\ll n$,
$m_{j}\ll k_{j}$ for $1\le j\le d$, and $k_{j+1}\ll m_{j}$ for
$1\le j<d$.

Since $\dim\mu<d$, we have $\sum_{l=1}^{d}\chi_{l}>\kappa$. By Proposition
\ref{prop:full ent of slices}, there exists $1\le j\le d$ so that
$\lambda_{n}(\mathrm{Q}_{j})>\delta_{0}$, where $\mathrm{Q}_{j}$
is the set of all integers $1\le i\le n$ such that
\[
\frac{1}{m_{j}}H\left(\theta^{*k_{j}},\mathcal{E}_{i+m_{j}}\mid\mathcal{E}_{i}\vee\pi_{[d]\setminus\{j\}}^{-1}\mathcal{E}_{i+m_{j}}\right)>\chi_{j}-\frac{1}{2}\left(\sum_{l=1}^{d}\chi_{l}-\kappa\right).
\]

Let $0<\eta<1$ be with $\delta_{0}^{-1}\ll\eta^{-1}\ll m_{d}$. By
Lemma \ref{lem:ent=00003Davg of loc ent} and since $\mathrm{diam}(\mathrm{supp}(\theta^{*k_{j}}*\mu))=O(k_{j}\epsilon^{-1})$,
\[
\frac{1}{n}H(\theta^{*k_{j}}*\mu,\mathcal{E}_{n})=\mathbb{E}_{1\le i\le n}\left(\frac{1}{m_{j}}H\left(\theta^{*k_{j}}*\mu,\mathcal{E}_{i+m_{j}}\mid\mathcal{E}_{i}\right)\right)+O\left(\frac{m_{j}+\log(k_{j}\epsilon^{-1})}{n}\right).
\]
From this, by the concavity of conditional entropy, by (\ref{eq:comens part after trans}),
and since $m_{j},k_{j},\epsilon^{-1},\eta^{-1}\ll n$ and $\eta^{-1}\ll m_{j}$,
\begin{eqnarray}
\frac{1}{n}H(\theta^{*k_{j}}*\mu,\mathcal{E}_{n}) & \ge & \lambda_{n}([n]\setminus\mathrm{Q}_{j})\mathbb{E}_{i\in[n]\setminus\mathrm{Q}_{j}}\left(\frac{1}{m_{j}}H\left(\mu,\mathcal{E}_{i+m_{j}}\mid\mathcal{E}_{i}\right)\right)\nonumber \\
 & + & \lambda_{n}(\mathrm{Q}_{j})\mathbb{E}_{i\in\mathrm{Q}_{j}}\left(\frac{1}{m_{j}}H\left(\theta^{*k_{j}}*\mu,\mathcal{E}_{i+m_{j}}\mid\mathcal{E}_{i}\right)\right)-\eta.\label{eq:lb split Q_j and not}
\end{eqnarray}

Let $i\in\mathrm{Q}_{j}$ be given. By the conditional entropy formula
(\ref{eq:extended cond ent form}), by concavity, and from (\ref{eq:comens part after trans}),
\begin{eqnarray*}
\frac{1}{m_{j}}H\left(\theta^{*k_{j}}*\mu,\mathcal{E}_{i+m_{j}}\mid\mathcal{E}_{i}\right) & \ge & \frac{1}{m_{j}}H\left(\mu,\pi_{[d]\setminus\{j\}}^{-1}\mathcal{E}_{i+m_{j}}\mid\mathcal{E}_{i}\right)\\
 & + & \frac{1}{m_{j}}H\left(\theta^{*k_{j}},\mathcal{E}_{i+m_{j}}\mid\mathcal{E}_{i}\vee\pi_{[d]\setminus\{j\}}^{-1}\mathcal{E}_{i+m_{j}}\right)-\eta.
\end{eqnarray*}
By Lemma \ref{lem:lb on proj of comp},
\[
\frac{1}{m_{j}}H\left(\mu,\pi_{[d]\setminus\{j\}}^{-1}\mathcal{E}_{i+m_{j}}\mid\mathcal{E}_{i}\right)\ge\sum_{j'\in[d]\setminus\{j\}}\chi_{j'}-\eta.
\]
From the last two inequalities and by the definition of $\mathrm{Q}_{j}$,
\begin{equation}
\frac{1}{m_{j}}H\left(\theta^{*k_{j}}*\mu,\mathcal{E}_{i+m_{j}}\mid\mathcal{E}_{i}\right)\ge\frac{1}{2}\left(\sum_{j'\in[d]}\chi_{j'}+\kappa\right)-2\eta\text{ for }i\in\mathrm{Q}_{j}.\label{eq:lb all i in Q_j}
\end{equation}

By Lemma \ref{lem:lb on ent of comp},
\[
\frac{1}{m_{j}}H\left(\mu,\mathcal{E}_{i+m_{j}}\mid\mathcal{E}_{i}\right)\ge\kappa-\eta\text{ for }i\in\mathbb{Z}_{>0}.
\]
Thus, from (\ref{eq:lb split Q_j and not}) and (\ref{eq:lb all i in Q_j}),
and since $\lambda_{n}(\mathrm{Q}_{j})>\delta_{0}$,
\begin{eqnarray*}
\frac{1}{n}H(\theta^{*k_{j}}*\mu,\mathcal{E}_{n}) & \ge & \lambda_{n}\left([n]\setminus\mathrm{Q}_{j}\right)\kappa+\lambda_{n}\left(\mathrm{Q}_{j}\right)\frac{1}{2}\left(\sum_{j'\in[d]}\chi_{j'}+\kappa\right)-O(\eta)\\
 & > & \kappa+\frac{\delta_{0}}{2}\left(\sum_{j'\in[d]}\chi_{j'}-\kappa\right)-O(\eta).
\end{eqnarray*}
Hence, from $\sum_{j'\in[d]}\chi_{j'}>\kappa$ and $\epsilon^{-1}\ll\delta_{0}^{-1}\ll\eta^{-1}$,
\[
\frac{1}{n}H(\theta^{*k_{j}}*\mu,\mathcal{E}_{n})>\kappa+\delta_{0}^{2}.
\]

From the last inequality, by Lemma \ref{lem:lim of ent wrt non-conf part},
and since $\delta_{0}\ll n$,
\[
\frac{1}{n}H(\theta^{*k_{j}}*\mu,\mathcal{E}_{n})-\frac{1}{n}H(\mu,\mathcal{E}_{n})>\delta_{0}^{2}/2.
\]
Thus, from Corollary \ref{cor:from Kaimanovich=002013Vershik} and
since $k_{j},\delta_{0}^{-1}\ll n$,
\[
\frac{1}{n}H(\theta*\mu,\mathcal{E}_{n})\ge\frac{1}{n}H(\mu,\mathcal{E}_{n})+\frac{\delta_{0}^{2}}{4k_{j}}.
\]
This, together with another application of Lemma \ref{lem:lim of ent wrt non-conf part},
completes the proof of the theorem with $\delta=\frac{\delta_{0}^{2}}{8k_{j}}$.
\end{proof}

\section{\label{sec:Proof-of main thm for measures}Proof of the main result
for self-affine measures}

The purpose of this section is to prove Theorem \ref{thm:main result for measures}.

\subsection{\label{subsec:Ledrappier-Young-formula}Ledrappier--Young formula
for diagonal self-affine measures}

Denote by $\mathcal{B}$ the Borel $\sigma$-algebra of $\mathbb{R}^{d}$.
Given $\theta\in\mathcal{M}_{\mathrm{c}}(\mathbb{R}^{d})$ and a sub-$\sigma$-algebra
$\mathcal{A}$ of $\mathcal{B}$, write $\{\theta_{x}^{\mathcal{A}}\}_{x\in\mathbb{R}^{d}}$
for the disintegration of $\theta$ with respect to $\mathcal{A}$
(see e.g. \cite[Section 5.3]{EiWa}). For $J\subset[d]$ we write
$\mathcal{B}_{J}$ in place of $\pi_{J}^{-1}\mathcal{B}$.

Let $\beta:=p^{\mathbb{N}}$ be the Bernoulli measure on $\Lambda^{\mathbb{N}}$
corresponding to $p$. That is, $\beta([u])=p_{u}$ for $u\in\Lambda^{*}$,
where $[u]$ is the cylinder set corresponding to $u$. Write $\mathcal{P}:=\left\{ [i]\right\} _{i\in\Lambda}$
for the partition of $\Lambda^{\mathbb{N}}$ into first generation
cylinders. For each $J\subset[d]$ set
\[
\mathrm{H}_{J}:=H\left(\beta,\mathcal{P}\mid\Pi^{-1}\mathcal{B}_{J}\right),
\]
where recall from Section \ref{subsec:The-setup} that $\Pi$ is the
coding corresponding to $\Phi$. Note that $\mathrm{H}_{\emptyset}=H(p)$.

Recall that we assume that $\chi_{1}<...<\chi_{d}$. Taking this into
account, the following theorem follows directly from \cite[Theorem 1.4]{MR4557759}.
See also \cite[Theorem 2.11]{FH-dimension} for a somewhat simpler
statement, which does not involve the dimension of slices.
\begin{thm}
\label{thm:LY formula}Let $1\le j_{1}<...<j_{s}\le d$, for $0\le b\le s$
set $J_{b}:=\{j_{1},...,j_{b}\}$, and set $\theta:=\pi_{J_{s}}\mu$.
Then, for each $0\le k<l\le s$ and for $\theta$-a.e. $x$, the measure
$\pi_{J_{l}}\theta_{x}^{\mathcal{B}_{J_{k}}}$ is exact dimensional
with,
\[
\dim\pi_{J_{l}}\theta_{x}^{\mathcal{B}_{J_{k}}}=\sum_{b=k+1}^{l}\frac{\mathrm{H}_{J_{b-1}}-\mathrm{H}_{J_{b}}}{\chi_{j_{b}}}.
\]
\end{thm}

\begin{rem}
\label{rem:after LY formula}In the notation of Theorem \ref{thm:LY formula},
note that for each $1\le b\le s$ and for $\theta$-a.e. $x$ the
measure $\pi_{J_{b}}\theta_{x}^{\mathcal{B}_{J_{b-1}}}$ is supported
on a translate of the line $e_{j_{b}}\mathbb{R}$. Thus $\dim\pi_{J_{b}}\theta_{x}^{\mathcal{B}_{J_{b-1}}}\le1$,
which by the theorem implies that $0\le\mathrm{H}_{J_{b-1}}-\mathrm{H}_{J_{b}}\le\chi_{j_{b}}$.
\end{rem}

\subsection{Notations for affine maps}

Write $\mathrm{Aff}(d)$ for the group of invertible affine maps from
$\mathbb{R}^{d}$ into itself. Let $\mathrm{ev}_{0}:\mathrm{Aff}(d)\rightarrow\mathbb{R}^{d}$
be the evaluation mapping of $\mathrm{Aff}(d)$ at the point $0$.
That is, $\mathrm{ev}_{0}(\psi)=\psi(0)$ for $\psi\in\mathrm{Aff}(d)$.
Given $\psi\in\mathrm{Aff}(d)$ we denote its linear part by $A_{\psi}$.
Thus, $\psi(x)=A_{\psi}x+\mathrm{ev}_{0}(\psi)$ for $x\in\mathbb{R}^{d}$.
Let $\mathcal{L}_{\mathrm{Aff}(d)}$ be the partition of $\mathrm{Aff}(d)$
according to the linear part. That is, for $\psi_{1},\psi_{2}\in\mathrm{Aff}(d)$
we have $\mathcal{L}_{\mathrm{Aff}(d)}(\psi_{1})=\mathcal{L}_{\mathrm{Aff}(d)}(\psi_{2})$
if and only if $A_{\psi_{1}}=A_{\psi_{2}}$.

For each $n\ge1$ set
\[
\nu^{(n)}:=\sum_{u\in\Lambda^{n}}p_{u}\cdot\delta_{\varphi_{u}}\in\mathcal{M}_{\mathrm{c}}\left(\mathrm{Aff}(d)\right),
\]
where $\delta_{\varphi_{u}}$ is the Dirac mass at $\varphi_{u}$.
Note that since the linear parts of the maps in $\Phi$ commute,
\begin{equation}
\#\left\{ L\in\mathcal{L}_{\mathrm{Aff}(d)}\::\:\nu^{(n)}(L)>0\right\} \le(n+1)^{|\Lambda|}.\label{eq:ub on card by commut of lin parts}
\end{equation}

\subsection{Asymptotic entropies of $\nu^{(n)}$}

We shall need the following theorem which resembles \cite[Theorem 1.4]{Ho1}.
\begin{thm}
\label{thm:lim of cond ent =00003D 0}Suppose that $\dim\mu<d$ and
that $\dim\pi_{J}\mu=|J|$ for each proper subset $J$ of $[d]$.
Then for every $M>1$,
\[
\underset{n}{\lim}\:\frac{1}{n}H\left(\nu^{(n)},\mathrm{ev}_{0}^{-1}(\mathcal{E}_{Mn})\vee\mathcal{L}_{\mathrm{Aff}(d)}\mid\mathrm{ev}_{0}^{-1}(\mathcal{E}_{n})\right)=0.
\]
\end{thm}

\begin{proof}
Write \emph{$\mathcal{L}:=\mathcal{L}_{\mathrm{Aff}(d)}$}, and for
each $n\ge0$ set $\mathcal{E}_{n}':=\mathrm{ev}_{0}^{-1}(\mathcal{E}_{n})$.
Let $M>1$ be given, and assume by contradiction that
\[
\underset{n}{\limsup}\:\frac{1}{n}H\left(\nu^{(n)},\mathcal{E}_{Mn}'\vee\mathcal{L}\mid\mathcal{E}_{n}'\right)>0.
\]
By (\ref{eq:ub on card by commut of lin parts}),
\[
\underset{n}{\limsup}\:\frac{1}{n}H\left(\nu^{(n)},\mathcal{E}_{Mn}'\mid\mathcal{E}_{n}'\vee\mathcal{L}\right)>0.
\]
Thus, there exist $0<\epsilon<1$ and an infinite $\mathrm{Q}\subset\mathbb{N}$,
so that
\[
\nu^{(n)}\left\{ \varphi\::\:\frac{1}{n}H\left(\mathrm{ev}_{0}\nu_{\mathcal{E}_{n}'\vee\mathcal{L}(\varphi)}^{(n)},\mathcal{E}_{Mn}\right)>\epsilon\right\} >\epsilon\text{ for }n\in\mathrm{Q}.
\]

Let $0<\delta,\eta<1$ and $n\in\mathrm{Q}$ be such that,
\begin{equation}
\epsilon^{-1},M\ll\delta^{-1}\ll\eta^{-1}\ll n.\label{eq:order of params lim of cond ent}
\end{equation}
Write $F$ for the set of all $\varphi\in\mathrm{supp}(\nu^{(n)})$
for which there exist $\alpha_{1},...,\alpha_{d}\in\mathbb{R}$ so
that $A_{\varphi}=\mathrm{diag}(\alpha_{1},...,\alpha_{d})$ and,
\[
\left|-\frac{1}{n}\log|\alpha_{j}|-\chi_{j}\right|<\eta\text{ for }1\le j\le d.
\]
By the weak law of large numbers we may assume that $\nu^{(n)}(F)>1-\eta$.

From $\mu=\sum_{u\in\Lambda^{n}}p_{u}\cdot\varphi_{u}\mu$ we get,
\[
\mu=\int\mathrm{ev}_{0}\nu_{\mathcal{E}_{n}'\vee\mathcal{L}(\varphi)}^{(n)}*A_{\varphi}\mu\:\:d\nu^{(n)}(\varphi).
\]
By the definition of $F$ it follows that for each $\varphi\in F$,
\[
\log\#\left\{ E\in\mathcal{E}_{n}\::\:\mathrm{supp}\left(\mathrm{ev}_{0}\nu_{\mathcal{E}_{n}'\vee\mathcal{L}(\varphi)}^{(n)}*A_{\varphi}\mu\right)\cap E\ne\emptyset\right\} =O(n\eta).
\]
Moreover, for $\varphi\in F$ the partitions $A_{\varphi}^{-1}\mathcal{E}_{Mn}$
and $\mathcal{E}_{(M-1)n}$ are $2^{O(n\eta)}$-commensurable. From
these facts and by the concavity of conditional entropy,
\begin{eqnarray}
\frac{1}{n}H\left(\mu,\mathcal{E}_{Mn}\mid\mathcal{E}_{n}\right) & \ge & \int_{F}\frac{1}{n}H\left(\mathrm{ev}_{0}\nu_{\mathcal{E}_{n}'\vee\mathcal{L}(\varphi)}^{(n)}*A_{\varphi}\mu,\mathcal{E}_{Mn}\right)\:d\nu^{(n)}(\varphi)-O(\eta)\nonumber \\
 & = & \int_{F}\frac{1}{n}H\left(A_{\varphi}^{-1}\mathrm{ev}_{0}\nu_{\mathcal{E}_{n}'\vee\mathcal{L}(\varphi)}^{(n)}*\mu,\mathcal{E}_{(M-1)n}\right)\:d\nu^{(n)}(\varphi)-O(\eta).\label{eq:first lb}
\end{eqnarray}
Note that from (\ref{eq:comens part after trans}), by the concavity
of entropy, and by Lemma \ref{lem:lim of ent wrt non-conf part},
it follows that for every $\sigma\in\mathcal{M}_{\mathrm{c}}(\mathbb{R}^{d})$,
\begin{equation}
\frac{1}{n}H\left(\sigma*\mu,\mathcal{E}_{(M-1)n}\right)\ge(M-1)\kappa-\eta.\label{eq:ent lb of conv with gen sig}
\end{equation}

Write $Z$ for the set of all $\varphi\in F$ for which,
\[
\frac{1}{n}H\left(\mathrm{ev}_{0}\nu_{\mathcal{E}_{n}'\vee\mathcal{L}(\varphi)}^{(n)},\mathcal{E}_{Mn}\right)>\epsilon.
\]
Since $n\in\mathrm{Q}$ and $\nu^{(n)}(F)>1-\eta$, we have $\nu^{(n)}(Z)\ge\epsilon-\eta>\epsilon/2$.
Fix $\varphi\in Z$ and set $\theta:=A_{\varphi}^{-1}\mathrm{ev}_{0}\nu_{\mathcal{E}_{n}'\vee\mathcal{L}(\varphi)}^{(n)}$.
From $\varphi\in Z$ and since $A_{\varphi}^{-1}\mathcal{E}_{Mn}$
and $\mathcal{E}_{(M-1)n}$ are $2^{O(n\eta)}$-commensurable,
\[
\frac{1}{n}H\left(\theta,\mathcal{E}_{(M-1)n}\right)>\epsilon-O(\eta).
\]
From $\varphi\in F$,
\[
\log\#\left\{ E\in\mathcal{E}_{0}\::\:\mathrm{supp}(\theta)\cap E\ne\emptyset\right\} =O(n\eta).
\]
Thus, by the convexity bound (\ref{eq:conc =000026 almo conv of ent}),
\begin{equation}
\mathbb{E}_{i=0}\left(\frac{1}{n}H\left(\theta_{x,i},\mathcal{E}_{(M-1)n}\right)\right)>\epsilon-O(\eta)>\epsilon/2.\label{eq:lb avg ent lev-0 ent}
\end{equation}

For each $x\in\mathrm{supp}(\theta)$
\[
\frac{1}{n}H\left(\theta_{x,0},\mathcal{E}_{(M-1)n}\right)\le CM,
\]
where $C>1$ is a global constant. Hence, setting
\[
Y:=\left\{ x\in\mathbb{R}^{d}\::\:\frac{1}{n}H\left(\theta_{x,0},\mathcal{E}_{(M-1)n}\right)>\frac{\epsilon}{4}\right\} ,
\]
it follows by (\ref{eq:lb avg ent lev-0 ent}) that $\theta(Y)\ge\frac{\epsilon}{4CM}$.
From this, by the concavity of entropy, from (\ref{eq:ent lb of conv with gen sig}),
and by Theorem \ref{thm:ent inc result},
\begin{eqnarray*}
\frac{1}{n}H\left(\theta*\mu,\mathcal{E}_{(M-1)n}\right) & \ge & \mathbb{E}_{i=0}\left(\frac{1}{n}H\left(\theta_{x,i}*\mu,\mathcal{E}_{(M-1)n}\right)\right)\\
 & \ge & \theta(Y)(M-1)(\kappa+\delta)+\theta(\mathbb{R}^{d}\setminus Y)((M-1)\kappa-\eta)\\
 & \ge & (M-1)\kappa+\frac{\epsilon(M-1)}{4CM}\delta-\eta.
\end{eqnarray*}

Note that the last estimate holds for all $\varphi\in Z$. From this,
$\nu^{(n)}(F)>1-\eta$, $\nu^{(n)}(Z)>\epsilon/2$, (\ref{eq:first lb})
and (\ref{eq:ent lb of conv with gen sig}), we get
\begin{eqnarray}
\frac{1}{n}H\left(\mu,\mathcal{E}_{Mn}\mid\mathcal{E}_{n}\right) & \ge & \nu^{(n)}(Z)\left((M-1)\kappa+\frac{\epsilon(M-1)}{4CM}\delta-\eta\right)\nonumber \\
 & + & \nu^{(n)}(F\setminus Z)\left((M-1)\kappa-\eta\right)-O(\eta)\label{eq:cl to end of cond thm}\\
 & \ge & (1-\eta)(M-1)\kappa+\frac{\epsilon^{2}(M-1)}{8CM}\delta-O(\eta).\nonumber 
\end{eqnarray}

On the other hand by Lemma \ref{lem:lim of ent wrt non-conf part},
\[
\frac{1}{n}H\left(\mu,\mathcal{E}_{Mn}\mid\mathcal{E}_{n}\right)=\frac{1}{n}H\left(\mu,\mathcal{E}_{Mn}\right)-\frac{1}{n}H\left(\mu,\mathcal{E}_{n}\right)\le(M-1)\kappa+\eta.
\]
This together with (\ref{eq:cl to end of cond thm}) gives,
\[
\delta=O\left(\frac{M^{2}\eta}{\epsilon^{2}(M-1)}\right).
\]
But this contradicts (\ref{eq:order of params lim of cond ent}),
which completes the proof of the theorem.
\end{proof}
We shall also need the following lemma.
\begin{lem}
\label{lem:ent of ev nu is close to ent of mu}We have,
\[
\underset{n}{\lim}\:\left|\frac{1}{n}H\left(\mu,\mathcal{E}_{n}\right)-\frac{1}{n}H\left(\mathrm{ev}_{0}\nu^{(n)},\mathcal{E}_{n}\right)\right|=0.
\]
\end{lem}

\begin{proof}
Let $0<\epsilon<1$ and let $n\ge1$ be with $\epsilon^{-1}\ll n$.
Given $(\omega_{k})_{k\ge0}=\omega\in\Lambda^{\mathbb{N}}$, write
$\omega|_{n}:=\omega_{0}...\omega_{n-1}\in\Lambda^{n}$ for the $n$th
prefix of $\omega$. Let $F$ be the set of all $\omega\in\Lambda^{\mathbb{N}}$
for which,
\[
\left|-\frac{1}{n}\log|r_{\omega|_{n},j}|-\chi_{j}\right|<\epsilon\text{ for }1\le j\le d.
\]
By the weak law of large numbers we may assume that $\beta(F)>1-\epsilon$,
where recall that $\beta$ is the Bernoulli measure corresponding
to $p$.

Let $\Pi_{n}:\Lambda^{\mathbb{N}}\rightarrow\mathbb{R}^{d}$ be with
$\Pi_{n}(\omega)=\varphi_{\omega|_{n}}(0)$ for $\omega\in\Lambda^{\mathbb{N}}$.
Note that,
\[
\mathrm{ev}_{0}\nu^{(n)}=\Pi_{n}\beta=\beta(F)\Pi_{n}\beta_{F}+\beta(F^{c})\Pi_{n}\beta_{F^{c}}.
\]
Moreover,
\[
\frac{\beta(F^{c})}{n}H\left(\Pi_{n}\beta_{F^{c}},\mathcal{E}_{n}\right)=O(\epsilon).
\]
Thus, by the concavity and almost convexity of entropy,
\begin{equation}
\frac{\beta(F)}{n}H\left(\Pi_{n}\beta_{F},\mathcal{E}_{n}\right)\le\frac{1}{n}H\left(\mathrm{ev}_{0}\nu^{(n)},\mathcal{E}_{n}\right)\le\frac{\beta(F)}{n}H\left(\Pi_{n}\beta_{F},\mathcal{E}_{n}\right)+O(\epsilon).\label{eq:close to push =0000231}
\end{equation}
Since
\[
\mu=\Pi\beta=\beta(F)\Pi\beta_{F}+\beta(F^{c})\Pi\beta_{F^{c}},
\]
it follows in the same manner that
\begin{equation}
\frac{\beta(F)}{n}H\left(\Pi\beta_{F},\mathcal{E}_{n}\right)\le\frac{1}{n}H\left(\mu,\mathcal{E}_{n}\right)\le\frac{\beta(F)}{n}H\left(\Pi\beta_{F},\mathcal{E}_{n}\right)+O(\epsilon).\label{eq:close to push =0000232}
\end{equation}

For $(\omega_{k})_{k\ge0}=\omega\in F$ and $1\le j\le d$,
\begin{multline*}
\left|\pi_{j}\left(\Pi_{n}\omega-\Pi\omega\right)\right|=\left|\pi_{j}\left(\varphi_{\omega|_{n}}(0)-\varphi_{\omega|_{n}}\left(\Pi(\omega_{k})_{k\ge n}\right)\right)\right|\\
=\left|r_{\omega|_{n},j}\pi_{j}\left(\Pi(\omega_{k})_{k\ge n}\right)\right|=O\left(2^{\epsilon n}2^{-\chi_{j}n}\right).
\end{multline*}
Thus by (\ref{eq:comens part under inv img}),
\[
\left|\frac{1}{n}H\left(\Pi_{n}\beta_{F},\mathcal{E}_{n}\right)-\frac{1}{n}H\left(\Pi\beta_{F},\mathcal{E}_{n}\right)\right|=O(\epsilon).
\]
From this, (\ref{eq:close to push =0000231}) and (\ref{eq:close to push =0000232}),
we get
\[
\left|\frac{1}{n}H\left(\mathrm{ev}_{0}\nu^{(n)},\mathcal{E}_{n}\right)-\frac{1}{n}H\left(\mu,\mathcal{E}_{n}\right)\right|=O(\epsilon),
\]
which completes the proof of the lemma.
\end{proof}

\subsection{\label{subsec:Proof-of-main Theorem for measures}Proof of Theorem
\ref{thm:main result for measures}}

First, let us provide the definition of the Lyapunov dimension $\dim_{L}(\Phi,p)$.
Recall that we assume that $\chi_{1}<...<\chi_{d}$, and set
\[
m(\Phi,p):=\max\left\{ 0\le j\le d\::\:\chi_{1}+...+\chi_{j}\le H(p)\right\} .
\]
With $m=m(\Phi,p)$, the Lyapunov dimension is defined by
\[
\dim_{L}(\Phi,p):=\begin{cases}
m+\frac{H(p)-\chi_{1}-...-\chi_{m}}{\chi_{m+1}} & ,\text{ if }m<d\\
d\frac{H(p)}{\chi_{1}+...+\chi_{d}} & ,\text{ if }m=d
\end{cases}.
\]

For the proof of the theorem we shall need the following two simple
lemmas, whose proof we omit.
\begin{lem}
\label{lem:max of f on S}Let $1\le j_{1}<...<j_{s}\le d$, write
\[
Y:=\left\{ (y_{1},...,y_{s})\in\mathbb{R}^{s}\::\:0\le y_{k}\le\chi_{j_{k}}\text{ for }1\le k\le s\text{ and }\sum_{k=1}^{s}y_{k}\le H(p)\right\} ,
\]
and let $f:Y\rightarrow[0,\infty)$ be with $f(y)=\sum_{k=1}^{s}y_{k}/\chi_{j_{k}}$
for $(y_{1},...,y_{s})=y\in Y$. Suppose that $H(p)<\sum_{k=1}^{s}\chi_{j_{k}}$,
and let $0\le m<s$ be with $\sum_{k=1}^{m}\chi_{j_{k}}\le H(p)<\sum_{k=1}^{m+1}\chi_{j_{k}}$.
Then
\[
\underset{y\in Y}{\max}\:f(y)=m+\frac{H(p)-\chi_{j_{1}}-...-\chi_{j_{m}}}{\chi_{j_{m+1}}},
\]
and
\[
\tilde{y}:=(\chi_{j_{1}},...,\chi_{j_{m}},H(p)-\chi_{j_{1}}-...-\chi_{j_{m}},0,...,0)
\]
is the unique point in $Y$ which satisfies $f(\tilde{y})=\underset{y\in Y}{\max}\:f(y)$.
\end{lem}

\begin{lem}
\label{lem:ub on min=00007Bdim_LY,d=00007D}For each $0\le m<d$ we
have,
\[
m+\frac{H(p)-\chi_{1}-...-\chi_{m}}{\chi_{m+1}}\ge\min\left\{ \dim_{L}(\Phi,p),d\right\} .
\]
\end{lem}

We can now prove Theorem \ref{thm:main result for measures}, whose
statement we first recall.
\begin{thm*}
Suppose that,
\begin{enumerate}
\item $\chi_{j}<\chi_{j+1}$ for each $1\le j<d$;
\item $\Phi_{j}$ is exponentially separated for each $1\le j\le d$;
\item the linear parts of $\Phi$ are contained in a $1$-dimensional subgroup.
\end{enumerate}
Then $\dim\mu=\min\left\{ \dim_{L}(\Phi,p),d\right\} $.
\end{thm*}
\begin{proof}
The proof is carried out by induction on $d$. Thus, assume that the
theorem holds whenever the dimension of the ambient space is strictly
less than $d$ (for $d=1$ this assumption is vacuous). Since $\dim_{L}(\Phi,p)$
is always an upper bound for $\dim\mu$, we only need to show that
$\dim\mu\ge\min\left\{ \dim_{L}(\Phi,p),d\right\} $.

Given $\emptyset\ne J\subset[d]$ recall the notation $\Phi_{J}$
from Section \ref{subsec:The-setup}. Note that the conditions of
the theorem are satisfied for the pair $(\Phi_{J},p)$, and that $\pi_{J}\mu$
is the self-affine measure corresponding to $\Phi_{J}$ and $p$.
Thus, by the induction hypothesis,
\begin{equation}
\dim\pi_{J}\mu=\min\left\{ \dim_{L}\left(\Phi_{J},p\right),|J|\right\} \text{ for }J\subset[d]\text{ with }0<|J|<d.\label{eq:by ind hyp}
\end{equation}

First suppose that $\dim\pi_{[d-1]}\mu<d-1$, and set $L:=\dim_{L}\left(\Phi_{[d-1]},p\right)$.
From (\ref{eq:by ind hyp}) and $\dim\pi_{[d-1]}\mu<d-1$, it follows
that $\dim\pi_{[d-1]}\mu=L$. In particular $L<d-1$, and so by the
definition of the Lyapunov dimension we have $L=\dim_{L}(\Phi,p)$.
Thus,
\[
\dim\mu\ge\dim\pi_{[d-1]}\mu=\dim_{L}(\Phi,p),
\]
which proves the theorem in the present case.

Next, suppose that $\dim\pi_{[d-1]}\mu=d-1$ and there exists a proper
nonempty subset $J$ of $[d]$ so that $\dim\pi_{J}\mu<|J|$. Set
$L:=\dim_{L}\left(\Phi_{J},p\right)$, write $s:=|J|$, and let $1\le j_{1}<...<j_{s}\le d$
be with $J=\{j_{1},...,j_{s}\}$. From (\ref{eq:by ind hyp}) and
$\dim\pi_{J}\mu<s$, it follows that $\dim\pi_{J}\mu=L$. In particular
$L<s$, and so by the definition of the Lyapunov dimension there exists
$0\le m<s$ so that
\[
\sum_{k=1}^{m}\chi_{j_{k}}\le H(p)<\sum_{k=1}^{m+1}\chi_{j_{k}}\text{ and }L=m+\frac{H(p)-\chi_{j_{1}}-...-\chi_{j_{m}}}{\chi_{j_{m+1}}}.
\]

For each $1\le k\le s$ set $J_{k}:=\{j_{1},...,j_{k}\}$ and $\Delta_{k}:=\mathrm{H}_{J_{k-1}}-\mathrm{H}_{J_{k}}$,
where $J_{0}:=\emptyset$ and $\mathrm{H}_{J_{0}},...,\mathrm{H}_{J_{s}}$
are defined in Section \ref{subsec:Ledrappier-Young-formula}. By
Theorem \ref{thm:LY formula},
\[
\sum_{k=1}^{s}\frac{\Delta_{k}}{\chi_{j_{k}}}=\dim\pi_{J}\mu=L=m+\frac{H(p)-\chi_{j_{1}}-...-\chi_{j_{m}}}{\chi_{j_{m+1}}}.
\]
By Remark \ref{rem:after LY formula}, we have $0\le\Delta_{k}\le\chi_{j_{k}}$
for $1\le k\le s$. Since $\mathrm{H}_{J_{0}}=H(p)$,
\begin{equation}
H(p)=\mathrm{H}_{J}+\sum_{k=1}^{s}\Delta_{k}\ge\sum_{k=1}^{s}\Delta_{k}.\label{eq:H(p)=00003DH_J_d-1+Sum of dim}
\end{equation}
By combining all of this with Lemma \ref{lem:max of f on S}, it follows
that $\Delta_{k}=\chi_{j_{k}}$ for $1\le k\le m$, $\Delta_{m+1}=H(p)-\sum_{k=1}^{m}\chi_{j_{k}}$,
and $\Delta_{k}=0$ for $m+1<k\le s$. Hence $H(p)=\sum_{k=1}^{s}\Delta_{k}$,
and so $\mathrm{H}_{J}=0$ by (\ref{eq:H(p)=00003DH_J_d-1+Sum of dim}).
Thus, since $\pi_{J}^{-1}(\mathcal{B})\subset\mathcal{B}$,
\[
\mathrm{H}_{[d]}=H\left(\beta,\mathcal{P}\mid\Pi^{-1}\mathcal{B}\right)\le H\left(\beta,\mathcal{P}\mid\Pi^{-1}\mathcal{B}_{J}\right)=\mathrm{H}_{J}=0.
\]

From $\dim\pi_{[d-1]}\mu=d-1$, Theorem \ref{thm:LY formula}, and
Remark \ref{rem:after LY formula},
\[
\mathrm{H}_{[k-1]}-\mathrm{H}_{[k]}=\chi_{k}\text{ for }1\le k<d.
\]
Thus, from $\mathrm{H}_{[d]}=0$ and $\mathrm{H}_{[0]}=H(p)$,
\[
\mathrm{H}_{[d-1]}-\mathrm{H}_{[d]}=H(p)-\sum_{k=1}^{d-1}\chi_{k}.
\]
From Theorem \ref{thm:LY formula} and by the last two equations,
\[
\dim\mu=\sum_{k=1}^{d}\frac{\mathrm{H}_{[k-1]}-\mathrm{H}_{[k]}}{\chi_{k}}=d-1+\frac{H(p)-\chi_{1}-...-\chi_{d-1}}{\chi_{d}}.
\]
This together with Lemma \ref{lem:ub on min=00007Bdim_LY,d=00007D}
completes the proof of the theorem also in the present case.

Finally, suppose that $\dim\pi_{J}\mu=|J|$ for each proper subset
$J$ of $[d]$. If $\dim\mu=d$ then there is nothing to prove, and
so we may assume that $\dim\mu<d$. Since the IFSs $\Phi_{1},...,\Phi_{d}$
are exponentially separated, there exists $M>1$ so that 
\[
\frac{1}{n}H\left(\nu^{(n)},\mathrm{ev}_{0}^{-1}(\mathcal{E}_{Mn})\vee\mathcal{L}_{\mathrm{Aff}(d)}\right)=H(p)
\]
for infinitely many integers $n\ge1$. From this and by Theorem \ref{thm:lim of cond ent =00003D 0},
it follows that there exists an increasing sequence $\{n_{k}\}_{k\ge1}\subset\mathbb{Z}_{>0}$
such that,
\[
\underset{k}{\lim}\:\frac{1}{n_{k}}H\left(\nu^{(n_{k})},\mathrm{ev}_{0}^{-1}(\mathcal{E}_{n_{k}})\right)=H(p).
\]
Thus, from Lemmata \ref{lem:ent of ev nu is close to ent of mu} and
\ref{lem:lim of ent wrt non-conf part}, we get $\kappa=H(p)$. By
dividing this equality by $\chi_{d}$ and recalling that $\kappa=\chi_{d}\dim\mu-\sum_{j=1}^{d-1}(\chi_{d}-\chi_{j})$,
we obtain
\[
\dim\mu=d-1+\frac{H(p)-\chi_{1}-...-\chi_{d-1}}{\chi_{d}}.
\]
This together with Lemma \ref{lem:ub on min=00007Bdim_LY,d=00007D}
completes the proof of the theorem.
\end{proof}

\section{\label{sec:Proof-of-main thm for sets}Proof of the main result for
self-affine sets}

The purpose of this section is to prove Theorem \ref{thm:main result for sets}.
First, let us provide the definition of the affinity dimension. We
give a simplified version of the definition, which is valid only in
the case of diagonal systems. For the general definition we refer
the reader to \cite[Proposition 4.1]{falconer1988hausdorff}.

Write $\mathrm{S}_{d}$ for the symmetric group over the set $[d]$.
For $\sigma\in\mathrm{S}_{d}$ and $s\ge0$ define $\phi_{\sigma}^{s}:\Lambda\rightarrow\mathbb{R}_{>0}$
by
\[
\phi_{\sigma}^{s}(i):=\begin{cases}
|r_{i,\sigma(1)}|\cdot...\cdot|r_{i,\sigma(m)}|\cdot|r_{i,\sigma(m+1)}|^{s-m} & \text{ if }s<d\\
|r_{i,1}\cdot...\cdot r_{i,d}|^{s/d} & \text{ if }s\ge d
\end{cases}\text{ for }i\in\Lambda,
\]
where $m:=\left\lfloor s\right\rfloor $. From \cite[Theorem 2.1]{MR3336332},
it follows that in our setting the affinity dimension $\dim_{A}\Phi$
can be defined as the unique $s\ge0$ for which,
\begin{equation}
\max_{\sigma\in\mathrm{S}_{d}}\sum_{i\in\Lambda}\phi_{\sigma}^{s}(i)=1.\label{eq:def of aff dim}
\end{equation}

We shall need the following three simple lemmas. The first lemma,
whose proof we omit, follows easily from the central limit theorem.
In what follows, given a finite nonempty index set $I$, a probability
vector $q:=(q_{i})_{i\in I}$, and $I'\subset I$, we write $q(I'):=\sum_{i\in I'}q_{i}$.
\begin{lem}
\label{lem:by CLT}Let $p=(p_{i})_{i\in\Lambda}$ be a strictly positive
probability vector, let $(c_{i})_{i\in\Lambda}=:c\in\mathbb{R}^{\Lambda}$
be with $c_{i}\ne0$ for some $i\in\Lambda$, and let $\epsilon>0$.
Then for every $n\ge N(p,c,\epsilon)\ge1$,
\[
p^{\times n}\left\{ i_{1}...i_{n}\in\Lambda^{n}\::\:\sum_{k=1}^{n}c_{i_{k}}=0\right\} <\epsilon,
\]
where $p^{\times n}:=(p_{u})_{u\in\Lambda^{n}}$.
\end{lem}

The second lemma, whose proof we also omit, follows easily from the
definition of the Lyapunov dimension.
\begin{lem}
\label{lem:cont of dim_L}Let $p=(p_{i})_{i\in\Lambda}$ be a probability
vector. Then for every $\epsilon>0$ there exists $\delta=\delta(p,\epsilon)>0$
so that the following holds. Let $n\ge1$ and let $q=(q_{u})_{u\in\Lambda^{n}}$
be a probability vector. Suppose that $\frac{1}{n}H(q)\ge H(p)-\delta$
and,
\[
\left|\frac{1}{n}\sum_{u\in\Lambda^{n}}q_{u}\log|r_{u,j}|-\sum_{i\in\Lambda}p_{i}\log|r_{i,j}|\right|\le\delta\text{ for each }1\le j\le d.
\]
Then,
\[
\dim_{L}\left(\{\varphi_{u}\}_{u\in\Lambda^{n}},q\right)>\dim_{L}\left(\Phi,p\right)-\epsilon.
\]
\end{lem}

\begin{lem}
\label{lem:exp sep proceeds}Let $\Psi=\{\psi_{i}\}_{i\in\Lambda}$
be an affine IFS on $\mathbb{R}$ which is exponentially separated,
and let $n\ge1$ be given. Then $\{\psi_{u}\}_{u\in\Lambda^{n}}$
is also exponentially separated.
\end{lem}

\begin{proof}
By assumption, there exists $0<c<1$ and an infinite $\mathrm{Q}\subset\mathbb{Z}_{>0}$
so that $\rho(\psi_{u_{1}},\psi_{u_{2}})\ge c^{k}$ for all $k\in\mathrm{Q}$
and distinct $u_{1},u_{2}\in\Lambda^{k}$, where $\rho$ is defined
in Definition \ref{def:exp sep}. Write $\mathrm{Q}':=\left\{ \left\lfloor k/n\right\rfloor \::\:k\in\mathrm{Q}\right\} $,
which is an infinite set. Let $m\in\mathrm{Q}'$ and let $u_{1},u_{2}\in\Lambda^{mn}$
be distinct. There exits $k\in\mathrm{Q}$ with $m=\left\lfloor k/n\right\rfloor $,
which implies that $mn\le k<mn+n$. Fix some $w\in\Lambda^{k-mn}$.
Since $\psi_{w}$ is an affine contraction and since $wu_{1}$ and
$wu_{2}$ are distinct elements of $\Lambda^{k}$,
\[
\rho(\psi_{u_{1}},\psi_{u_{2}})\ge\rho(\psi_{w}\circ\psi_{u_{1}},\psi_{w}\circ\psi_{u_{2}})\ge c^{k}>c^{mn}c^{n}.
\]
This clearly shows that $\{\psi_{u}\}_{u\in\Lambda^{n}}$ is exponentially
separated, which completes the proof of the lemma.
\end{proof}
We can now prove Theorem \ref{thm:main result for sets}, whose statement
we first recall.
\begin{thm}
Suppose that,
\begin{enumerate}
\item for each $1\le j_{1}<j_{2}\le d$ there exists $i\in\Lambda$ so that
$|r_{i,j_{1}}|\ne|r_{i,j_{2}}|$;
\item $\Phi_{j}$ is exponentially separated for each $1\le j\le d$.
\end{enumerate}
Then $\dim_{H}K_{\Phi}=\min\left\{ \dim_{A}\Phi,d\right\} $.
\end{thm}

\begin{proof}
Since $\dim_{A}\Phi$ is always an upper bound for $\dim_{H}K_{\Phi}$,
we only need to show that $\dim_{H}K_{\Phi}\ge\min\left\{ \dim_{A}\Phi,d\right\} $.
Set $s:=\dim_{A}\Phi$ and write $\phi^{s}$ in place of $\phi_{e}^{s}$,
where $e$ denotes the identity element of $\mathrm{S}_{d}$ here.
Assume without loss of generality that,
\begin{equation}
\sum_{i\in\Lambda}\phi^{s}(i)\ge\sum_{i\in\Lambda}\phi_{\sigma}^{s}(i)\text{ for all }\sigma\in\mathrm{S}_{d}.\label{eq:phi^s is largest}
\end{equation}
Otherwise, we can simply permute the coordinates of the ambient space
$\mathbb{R}^{d}$ to adjust this. Note that from (\ref{eq:def of aff dim})
and (\ref{eq:phi^s is largest}) it follows that $\sum_{i\in\Lambda}\phi^{s}(i)=1$.

Write $p_{i}:=\phi^{s}(i)$ for $i\in\Lambda$, so that $p:=(p_{i})_{i\in\Lambda}$
is a probability vector. Set $m:=\left\lfloor s\right\rfloor $ and
note that for $i\in\Lambda$,
\begin{equation}
p_{i}=\begin{cases}
|r_{i,1}|\cdot...\cdot|r_{i,m}|\cdot|r_{i,m+1}|^{s-m} & \text{ if }s<d\\
|r_{i,1}\cdot...\cdot r_{i,d}|^{s/d} & \text{ if }s\ge d
\end{cases}.\label{eq:def of p_i}
\end{equation}
For $1\le j\le d$ set $\chi_{j}:=-\sum_{i\in\Lambda}p_{i}\log|r_{i,j}|$. 

Let us show that $s=\dim_{L}(\Phi,p)$. Assuming $s\ge d$,
\[
H(p)=-\sum_{i\in\Lambda}p_{i}\log|r_{i,1}\cdot...\cdot r_{i,d}|^{s/d}=\frac{s}{d}\sum_{j=1}^{d}\chi_{j}\ge\sum_{j=1}^{d}\chi_{j},
\]
and so,
\[
\dim_{L}(\Phi,p)=d\frac{H(p)}{\chi_{1}+...+\chi_{d}}=s.
\]
Suppose next that $m<s<d$. From the concavity of the $\log$ function
it follows that for $1\le j\le m$,
\[
(m+1-s)(\chi_{j}-\chi_{m+1})=\sum_{i\in\Lambda}p_{i}\log\frac{|r_{i,m+1}|^{m+1-s}}{|r_{i,j}|^{m+1-s}}\le\log\left(\sum_{i\in\Lambda}p_{i}\frac{|r_{i,m+1}|^{m+1-s}}{|r_{i,j}|^{m+1-s}}\right).
\]
From (\ref{eq:phi^s is largest}) and $\sum_{i\in\Lambda}\phi^{s}(i)=1$
we obtain
\[
\sum_{i\in\Lambda}p_{i}\frac{|r_{i,m+1}|^{m+1-s}}{|r_{i,j}|^{m+1-s}}\le1,
\]
and so $\chi_{j}\le\chi_{m+1}$ for $1\le j\le m$. A similar argument
shows that $\chi_{m+1}\le\chi_{j}$ for $m+1<j\le d$. Moreover, it
is easy to verify that,
\[
\sum_{j=1}^{m}\chi_{j}\le H(p)<\sum_{j=1}^{m+1}\chi_{j}.
\]
From all of this and by the definition of the Lyapunov dimension,
\[
\dim_{L}(\Phi,p)=m+\frac{H(p)-\chi_{1}-...-\chi_{m}}{\chi_{m+1}}=s.
\]
In a similar manner, the equality $s=\dim_{L}(\Phi,p)$ can be shown
to hold also in the case $m=s<d$, which shows that it holds in all
cases.

Let $0<\epsilon_{1},\epsilon_{2},\epsilon_{3}<1$ and $n\ge1$ be
with $\epsilon_{1}^{-1}\ll\epsilon_{2}^{-1}\ll\epsilon_{3}^{-1}\ll n$.
Write $p^{\times n}:=(p_{u})_{u\in\Lambda^{n}}$ and,
\[
\mathrm{R}:=\left\{ \left(r_{u,1},...,r_{u,d}\right)\::\:u\in\Lambda^{n}\right\} .
\]
For $r\in\mathrm{R}$ set,
\[
\mathcal{W}_{r}:=\left\{ u\in\Lambda^{n}\::\:\left(r_{u,1},...,r_{u,d}\right)=r\right\} \text{ and }\alpha_{r}:=p^{\times n}(\mathcal{W}_{r}).
\]
Note that $\alpha:=(\alpha_{r})_{r\in\mathrm{R}}$ is a probability
vector. Additionally, for $r\in\mathrm{R}$ let $q_{r}:=(q_{r,u})_{u\in\Lambda^{n}}$
be the probability vector with,
\[
q_{r,u}:=\begin{cases}
\frac{p_{u}}{\alpha_{r}} & \text{ if }u\in\mathcal{W}_{r}\\
0 & \text{ else}
\end{cases}\quad\text{ for }u\in\Lambda^{n}.
\]

From
\[
\mathrm{R}\subset\left\{ \left(\Pi_{i\in\Lambda}r_{i,1}^{k_{i,1}},...,\Pi_{i\in\Lambda}r_{i,d}^{k_{i,d}}\right)\::\:0\le k_{i,j}\le n\text{ for }i\in\Lambda\text{ and }1\le j\le d\right\} ,
\]
it follows that $|\mathrm{R}|\le(n+1)^{d|\Lambda|}$. Thus, by the
convexity bound (\ref{eq:conc =000026 almo conv of ent}) and since
$\epsilon_{2}^{-1}\ll n$,
\begin{multline}
\frac{1}{n}H\left(p^{\times n}\right)=\frac{1}{n}H\left(\sum_{r\in\mathrm{R}}\alpha_{r}q_{r}\right)\le\frac{1}{n}\sum_{r\in\mathrm{R}}\alpha_{r}H(q_{r})+\frac{1}{n}H(\alpha)\\
\le\frac{1}{n}\sum_{r\in\mathrm{R}}\alpha_{r}H(q_{r})+\frac{d|\Lambda|\log(n+1)}{n}\le\frac{1}{n}\sum_{r\in\mathrm{R}}\alpha_{r}H(q_{r})+\epsilon_{2}/2.\label{eq:by conv bd H(q) <=00003D}
\end{multline}

Write,
\[
\mathrm{R}_{1}:=\left\{ r\in\mathrm{R}\::\:\frac{H(p^{\times n})-H(q_{r})}{n}\le\epsilon_{2}\right\} .
\]
Note that $\frac{1}{n}\left(H(p^{\times n})-H(q_{r})\right)\ge-\log|\Lambda|$
for $r\in\mathrm{R}$. From this and (\ref{eq:by conv bd H(q) <=00003D}),
\begin{multline*}
\epsilon_{2}/2\ge\sum_{r\in\mathrm{R}}\alpha_{r}\frac{H(p^{\times n})-H(q_{r})}{n}\\
\ge\alpha\left(\mathrm{R}\setminus\mathrm{R}_{1}\right)\epsilon_{2}-\alpha\left(\mathrm{R}_{1}\right)\log|\Lambda|\ge\epsilon_{2}-\alpha\left(\mathrm{R}_{1}\right)\log\left(2|\Lambda|\right).
\end{multline*}
Hence,
\begin{equation}
\alpha\left(\mathrm{R}_{1}\right)\ge\epsilon_{2}/\left(2\log\left(2|\Lambda|\right)\right).\label{eq:alpha(R_1)>=00003D}
\end{equation}

By the weak law of large numbers and since $\epsilon_{2}^{-1},\epsilon_{3}^{-1}\ll n$,
\[
p^{\times n}\left\{ u\in\Lambda^{n}\::\:\left|\frac{1}{n}\log|r_{u,j}|+\chi_{j}\right|\ge\epsilon_{2}\right\} <\epsilon_{3}\text{ for }1\le j\le d.
\]
This clearly implies that,
\begin{equation}
\alpha\left\{ \left(r_{1},...,r_{d}\right)\in\mathrm{R}\::\left|\frac{1}{n}\log|r_{j}|+\chi_{j}\right|\ge\epsilon_{2}\text{ for some }1\le j\le d\right\} <d\epsilon_{3}.\label{eq:alpha=00007Bexp not close=00007D<=00003D}
\end{equation}

By condition (\ref{enu:cond reg mod of r_i,j}) in the statement of
the theorem, by Lemma \ref{lem:by CLT}, and since $\epsilon_{3}^{-1}\ll n$,
\[
p^{\times n}\left\{ u\in\Lambda^{n}\::\:|r_{u,j_{1}}|=|r_{u,j_{2}}|\right\} <\epsilon_{3}\text{ for }1\le j_{1}<j_{2}\le d.
\]
Thus,
\begin{equation}
\alpha\left\{ \left(r_{1},...,r_{d}\right)\in\mathrm{R}\::\:|r_{j_{1}}|=|r_{j_{2}}|\text{ for some }1\le j_{1}<j_{2}\le d\right\} <d^{2}\epsilon_{3}.\label{eq:alpha=00007Bcont eq=00007D<=00003D}
\end{equation}

From (\ref{eq:alpha(R_1)>=00003D}), (\ref{eq:alpha=00007Bexp not close=00007D<=00003D})
and (\ref{eq:alpha=00007Bcont eq=00007D<=00003D}), and since $\epsilon_{2}^{-1}\ll\epsilon_{3}^{-1}$,
it follows that there exists $\left(r_{1},...,r_{d}\right)=r\in\mathrm{R}$
so that $H(q_{r})\ge n\left(H(p)-\epsilon_{2}\right)$, $\left|\frac{1}{n}\log|r_{j}|+\chi_{j}\right|<\epsilon_{2}$
for each $1\le j\le d$, and $|r_{j_{1}}|\ne|r_{j_{2}}|$ for each
$1\le j_{1}<j_{2}\le d$. By Lemma \ref{lem:cont of dim_L} and since
$\epsilon_{1}^{-1}\ll\epsilon_{2}^{-1}$,
\[
\dim_{L}\left(\{\varphi_{u}\}_{u\in\Lambda^{n}},q_{r}\right)>\dim_{L}(\Phi,p)-\epsilon_{1}=\dim_{A}\Phi-\epsilon_{1}.
\]

Let $\mu$ be the self-affine measure corresponding to $\{\varphi_{u}\}_{u\in\Lambda^{n}}$
and $q_{r}$. Since $\mu$ is supported on $K_{\Phi}$, from Lemma
\ref{lem:exp sep proceeds}. and by Theorem \ref{thm:main result for measures},
\[
\dim_{H}K_{\Phi}\ge\dim\mu=\min\left\{ \dim_{L}\left(\{\varphi_{u}\}_{u\in\Lambda^{n}},q_{r}\right),d\right\} \ge\min\left\{ \dim_{A}\Phi-\epsilon_{1},d\right\} .
\]
Since $\epsilon_{1}$ is an arbitrarily small positive number, this
completes the proof of the theorem.
\end{proof}
\bibliographystyle{plain}
\bibliography{bibfile}

$\newline$\textsc{Department of Mathematics, Technion, Haifa, Israel}$\newline$$\newline$\textit{E-mail: }
\texttt{arapaport@technion.ac.il}
\end{document}